\documentclass{amsart}

\usepackage{time}

\usepackage[margin=1.5in]{geometry}

\usepackage{amscd}
\usepackage{amsmath}
\usepackage{amssymb}
\usepackage{amsthm}
\usepackage{arydshln}
\usepackage{bbm}
\usepackage{bm}
\usepackage{colonequals}
\usepackage{color}
\usepackage{enumitem}
\usepackage{latexsym}
\usepackage{mathrsfs}
\usepackage{mathtools}
\usepackage{stmaryrd}
\usepackage{textcomp}
\usepackage{tikz-cd}
\usepackage{url}
\usepackage{varwidth}
\usepackage[all]{xy}
\usepackage{yhmath}
\usepackage{soul}

\usepackage[colorlinks=false,hidelinks]{hyperref}
\usetikzlibrary{decorations.pathmorphing}

\newcommand{\ad}{\mathrm{ad}}

\newcommand{\der}{\mathrm{der}}

\newcommand{\pr}{\mathrm{pr}}

\newcommand{\cO}{\mathcal{O}}

\newcommand{\bfT}{\mathbf{T}}
\newcommand{\x}{\mathbf{x}}

\newcommand{\bbT}{\mathbb{T}}

\newcommand{\bbU}{\mathbb{U}}

\newcommand{\FF}{\mathbb{F}}
\newcommand{\bbF}{\mathbb{F}}

\newcommand{\bfP}{\mathbf{P}}
\newcommand{\bbP}{\mathbb{P}}
\newcommand{\bfM}{\mathbf{M}}
\newcommand{\bbA}{\mathbb{A}}

\newcommand{\bbR}{\mathbb{R}}

\newcommand{\bbQ}{\mathbb{Q}}
\newcommand{\QQ}{\mathbb{Q}}
\newcommand{\bfQ}{\mathbf{Q}}

\newcommand{\bbC}{\mathbb{C}}

\newcommand{\bbS}{\mathbb{S}}

\newcommand{\bbG}{\mathbb{G}}
\newcommand{\bfG}{\mathbf{G}}
\newcommand{\bbH}{\mathbb{H}}
\newcommand{\cH}{\mathcal{H}}
\newcommand{\bbK}{\mathbb{K}}
\newcommand{\bfK}{\mathbf{K}}
\newcommand{\bbV}{\mathbb{V}}

\newcommand{\bbB}{\mathbb{B}}
\newcommand{\bfB}{\mathbf{B}}
\newcommand{\bfN}{\mathbf{N}}
\newcommand{\bbN}{\mathbb{N}}
\newcommand{\bbM}{\mathbb{M}}
\newcommand{\bbL}{\mathbb{L}}
\newcommand{\bfL}{\mathbf{L}}

\newcommand{\cL}{\mathcal{L}}
\newcommand{\bfV}{\mathbf{V}}

\newcommand{\cR}{\mathcal{R}}

\newcommand{\bfH}{\mathbf{H}}

\newcommand{\bfU}{\mathbf{U}}

\newcommand{\cS}{\mathcal{S}}

\newcommand{\cN}{\mathcal N}
\newcommand{\cX}{\mathcal X}
\newcommand{\G}{\mathsf{G}}

\newcommand{\cG}{\mathcal{G}}
\newcommand{\bbW}{\mathbb{W}}
\newcommand{\triv}{\mathrm{triv}}

\DeclareMathOperator{\tr}{tr}

\DeclareMathOperator{\Ad}{Ad}

\DeclareMathOperator{\Ind}{Ind}

\DeclareMathOperator{\Stab}{Stab}

\DeclareMathOperator{\Tr}{Tr}

\DeclareMathOperator{\Char}{char}

\DeclareMathOperator{\GL}{GL}

\DeclareMathOperator{\GSp}{GSp}

\DeclareMathOperator{\Inf}{Inf}

\newcommand{\from}{\colon}

\theoremstyle{plain}
\newtheorem{thm}{Theorem}[section]
\newtheorem{theorem}[thm]{Theorem}
\newtheorem*{thm*}{Theorem}

\newtheorem{proposition}[thm]{Proposition}

\newtheorem{lemma}[thm]{Lemma}

\newtheorem{corollary}[thm]{Corollary}

\theoremstyle{definition}

\newtheorem{definition}[thm]{Definition}

\newtheorem{claim}{Claim}[thm]

\theoremstyle{remark}

\newtheorem*{claim*}{Claim}
\newtheorem{remark}[thm]{Remark}

\theoremstyle{theorem}
\newtheorem{displaytheorem}{Theorem}

\newtheorem*{maintheorem}{Main Theorem}
\newtheorem*{scalarconjecture*}{Scalar Product Conjecture}
\newtheorem*{displaycorollary}{Corollary}

\makeatletter
\newcommand{\dashover}[2][\mathop]{#1{\mathpalette\df@over{{\dashfill}{#2}}}}
\newcommand{\fillover}[2][\mathop]{#1{\mathpalette\df@over{{\solidfill}{#2}}}}
\newcommand{\df@over}[2]{\df@@over#1#2}
\newcommand\df@@over[3]{%
  \vbox{
    \offinterlineskip
    \ialign{##\cr
      #2{#1}\cr
      \noalign{\kern1pt}
      $\m@th#1#3$\cr
    }
  }%
}
\newcommand{\dashfill}[1]{%
  \kern-.5pt
  \xleaders\hbox{\kern.5pt\vrule height.4pt width \dash@width{#1}\kern.5pt}\hfill
  \kern-.5pt
}
\newcommand{\dash@width}[1]{%
  \ifx#1\displaystyle
    2pt
  \else
    \ifx#1\textstyle
      1.5pt
    \else
      \ifx#1\scriptstyle
        1.25pt
      \else
        \ifx#1\scriptscriptstyle
          1pt
        \fi
      \fi
    \fi
  \fi
}
\newcommand{\solidfill}[1]{\leaders\hrule\hfill}
\makeatother

\SelectTips{cm}{11}

\title{The scalar product formula \\  for parahoric Deligne--Lusztig induction}
\author{Charlotte Chan}
\address{Department of Mathematics, University of Michigan, 2074 East Hall, 530 Church Street, Ann Arbor, MI 48105, USA.}\email{charchan@umich.edu}

\begin{document}

%\subjclass[2010]{Primary: 22E50; Secondary: 11S37, 11F70}
%\keywords{supercuspidal representations, Deligne--Lusztig theory, Harish-Chandra character}

\maketitle

\begin{abstract}
  Parahoric Deligne--Lusztig induction gives rise to positive-depth representations of parahoric subgroups of $p$-adic groups. The most fundamental basic question about parahoric Deligne--Lusztig induction is whether it satisfies the scalar product formula. We resolve this conjecture for all Howe-factorizable split-generic pairs $(T,\theta)$---in particular, for all characters $\theta$ when $T$ is elliptic and $p$ is not a torsion prime for the root system of the $p$-adic group.  
\end{abstract}

\section{Introduction}\label{sec:introduction}

In the early 2000s, Lusztig established \cite{Lus04} an analogue of Deligne--Lusztig induction for algebraic groups arising as jet schemes $\bbG_r$ of connected reductive groups $\bbG$ over finite fields $\FF_q$. Lusztig defined, for the jet scheme $\bbT_r$ of any maximal torus $\bbT$ of $\bbG$, a functor
\begin{equation*}
  R_{\bbT_r, \bbB_r}^{\bbG_r} \from \cR(\bbT_r(\FF_q)) \to \cR(\bbG_r(\FF_q))
\end{equation*}
where $\bbB_r$ is the jet scheme of a Borel subgroup  over $\overline \FF_q$ which contains $\bbT_{\overline \FF_q}$ and $\cR$ denotes the Grothendieck ring. %We say that the scalar product formula holds for the triple $(T_r, B_r, \theta)$ if for $(T_r', B_r', \theta')$,
% \begin{equation}\label{eq:scalar product}
%   \langle R_{T_r,B_r}^{G_r}(\theta), R_{T_r',B_r'}^{G_r}(\theta') \rangle = \sum_{w \in W_{G_r}(T_r,T_r')(\FF_q)} \langle \theta, {}^w \theta' \rangle.
% \end{equation}
% Lusztig proved that $(T_r,B_r,\theta)$ satisfies the scalar product formula if $\theta$ satisfies a genericity condition. 
%This functor is defined by constructing a $\overline \FF_q$-scheme $X_{\bbT_r,\bbB_r}^{\bbG_r}$ with a $(\bbG_r(\FF_q) \times \bbT_r(\FF_q))$-action, considering Euler characteristic with respect to compactly supported \'etale cohomology, and taking isotypic components.
Lusztig proved that if a character $\theta \from \bbT_r(\FF_q) \to \overline \QQ_\ell^\times$ satisfies a strong genericity condition, then for any $(\theta',\bbT_r', \bbB_r')$, 
\begin{equation}\label{eq:scalar product}
  \langle R_{\bbT_r,\bbB_r}^{\bbG_r}(\theta), R_{\bbT_r',\bbB_r'}^{\bbG_r}(\theta') \rangle = \sum_{w \in W_{\bbG_r}(\bbT_r,\bbT_r')(\FF_q)} \langle \theta, {}^w \theta' \rangle.
\end{equation}
In particular, this formula proves that under the above genericity condition, $R_{\bbT_r,\bbB_r}^{\bbG_r}(\theta)$ is independent of the choice of $\bbB_r$ and that $R_{\bbT_r,\bbB_r}^{\bbG_r}(\theta)$ is irreducible if it has trivial stabilizer in the Weyl group. These results were extended by Stasinski \cite{Sta09} to mixed-characteristic jet schemes and by the author and Ivanov \cite{CI21-RT} to algebraic groups---also denoted by $\bbG_r$---arising from Moy--Prasad quotients of parahoric subgroup schemes associated to unramified maximal tori $T$ of connected reductive groups $G$ over non-archimedean local fields $F$. 

From the perspective of the representation theory of $p$-adic groups $G(F)$, this more general setting of $\bbG_r$ arising from parahoric subgroups is essential. When $T \subset G$ is elliptic, the author and Oi \cite{CO21} proved that under the aforementioned genericity condition and a largeness condition on $q$, the representations $R_{\bbT_r,\bbB_r}^{\bbG_r}(\theta)$ give rise to $L$-packets of toral supercuspidal representations in the sense of \cite{Ree08,DS18}. A serious obstruction to proving such a comparison result for regular supercuspidal representations \cite{Kal19} beyond the toral setting is establishing \eqref{eq:scalar product} in general, which is arguably the most fundamental basic question about the functor $R_{\bbT_r,\bbB_r}^{\bbG_r}$:

\begin{scalarconjecture*}
  Fix $(\theta,\bbT_r, \bbB_r)$. For all $(\theta',\bbT_r', \bbB_r')$, the formula \eqref{eq:scalar product} holds.
  % \begin{equation*}
  %   \langle R_{T_r,B_r}^{G_r}(\theta), R_{T_r',B_r'}^{G_r}(\theta') \rangle = \sum_{w \in W_{G_r}(T_r,T_r')(\FF_q)} \langle \theta, {}^w \theta' \rangle.
  % \end{equation*}
\end{scalarconjecture*}

When $r = 0$, it is a classical theorem of Deligne and Lusztig \cite{DL76} that the scalar product formula holds for all $(\theta,\bbT_0, \bbB_0)$. For $r>0$, this conjecture is obviously false as stated, the simplest example of which was discussed in \cite{CI21-RT}: when $\bbT_r$ is the jet scheme associated to the split torus of $\bbG$, then $R_{\bbT_r,\bbB_r}^{\bbG_r}(\theta) = \Ind_{\bbB_r(\FF_q)}^{\bbG_r(\FF_q)}(\tilde \theta)$ where $\tilde \theta = \theta \circ \pr$ for $\pr \from \bbB_r(\FF_q) \to \bbT_r(\FF_q)$. If $\theta$ factors through a character on $\bbT_0$ in general position, then $R_{\bbT_r,\bbB_r}^{\bbG_r}(\theta)$ is not irreducible for any $r > 0$.

As mentioned above, for $r>0$, thanks to \cite{Lus04,Sta09,CI21-RT} the Scalar Product Conjecture is known to be true for arbitrary (unramified maximal) $T$ if $\theta$ satisfies a genericity condition which we call \textit{weakly $(T,G)$-generic} (it is a strong nontriviality condition on the restriction of $\theta$ to $\ker(\bbT_r^\sigma \to \bbT_{r-1}^\sigma)$). Outside this setting, results are sparser: when $G$ is an inner form of $\GL_n$, this was proved by the author and Ivanov in \cite{CI_loopGLn}, the techniques of which were vastly generalized by work of Dudas and Ivanov in \cite{DI20}, which established the Scalar Product Conjecture for $T$ Coxeter under a mild root-theoretic assumption on $q$ ($q > 5$ suffices). In these works beyond weak $(T,G)$-genericity, $\bbB_r$ was forced to be chosen to be optimal for the methods involved.

In the present paper, we establish a novel approach and prove:

\begin{maintheorem}\label{thm:scalar product}
  If $p$ is not a torsion prime for the root system of $G$, then the Scalar Product Conjecture holds for all split-generic $(\theta,\bbT_r)$.
\end{maintheorem}

%The notion of split-genericity interpolates between $\theta$ being $(\bfT,\bfG)$-generic for $\bfT$ split and having no assumption on $\theta$ when $\bfT$ is elliptic. 
The condition on $p$ comes about because our approach allows us to establish the Scalar Product Conjecture whenever $(\theta,\bbT_r)$ is \textit{Howe-factorizable}. In general, from any $(\theta,\bbT_r)$, one can extract a nested sequence of subsets of roots. If $p$ is a torsion prime for the root system of $G$, then it can happen that these subsets are not Levi subsystems, in which case $(\theta, \bbT_r)$ does not have a Howe factorization. It is a result of Kaletha \cite{Kal19} that if $p$ is not a torsion prime, then \textit{every} $(\theta, \bbT_r)$ has a Howe factorization. The torsion primes of irreducible root systems are \cite[Corollary 1.13]{Ste75}:
\begin{center}
  \begin{tabular}{| c | c | c | c | c | c | c | c |} \hline
    & $B_n$ & $D_n$ & $E_6$ & $E_7$ & $E_8$ & $F_4$ & $G_2$ \\ \hline
    torsion primes & 2 & 2 & 2, 3 & 2, 3 & 2, 3, 5 & 2,3 & 2 \\ \hline
  \end{tabular}
\end{center}
For the exposition's sake, we will implicitly assume for the rest of the introduction that either: $p$ is not a torsion prime for $G$, or $(\theta,\bbT_r)$ has a Howe factorization.

Now let us say a few words about the notion of split-genericity. For a fixed maximal torus $T \subset G$, the proportion of split-generic characters $\theta$ of $\bbT_r$ depends on the ``degree of ellipticity'' of $T$: on one extreme, if $T$ is the split torus, then $\theta$ is split-generic if and only if it is $(T,G)$-generic, and on the other extreme, if $T$ is elliptic, then all $\theta$ are split-generic. We see therefore that the Main Theorem includes all previously known progress towards the Conjecture and also explains  the spectrum of dependence on $(\theta,\bbT_r)$. We expect, but cannot prove at present, that this result is sharp; precisely, we expect that if $(\theta,\bbT_r)$ is not split-generic, then there exists a triple $(\theta',\bbT_r',\bbB_r')$ for which the Scalar Product Formula does not hold. 

We note the following direct consequence of the Main Theorem: %We state a direct consequence of the Main Theorem.

\begin{displaycorollary}
  For $T$ elliptic, $R_{\bbT_r,\bbB_r}^{\bbG_r}(\theta)$ is 
  %independent of the choice of $\bbB_r$ and is 
  irreducible if and only if $\Stab_{W_{\bbG_r(\FF_q)}(\bbT_r)}(\theta) = \{1\}.$
\end{displaycorollary}

The new approach to the Scalar Product Conjecture presented in this paper is to describe $R_{\bbT_r,\bbB_r}^{\bbG_r}(\theta)$ in terms of a sequence of parahoric Lusztig inductions associated to a(ny) Howe factorization of $\theta$. The present paper appears to be the first work to observe this structure and also the first to define parahoric Lusztig induction, though both are almost obvious expectations \textit{a posteriori}. As the name suggests, parahoric Lusztig induction is a natural generalization of classical Lusztig induction \cite{Lus76} in the parahoric setting of \cite{CI21-RT}. We will need several general properties of parahoric Lusztig induction, which we establish in Section \ref{sec:parahoric lusztig}. Of these results, Proposition \ref{prop:twisting} is the most nontrivial (for example, see Remarks \ref{rem:small q} and \ref{rem:green insufficient} for discussions on subtleties and on a proof of this fact in $r=0$ that fails for $r>0$).

This description of $R_{\bbT_r,\bbB_r}^{\bbG_r}(\theta)$ in terms of parahoric Lusztig inductions will certainly illuminate the relationship between positive-depth Deligne--Lusztig induction and the author's recent work with Bezrukavnikov \cite{BC24} constructing generic character sheaves on parahoric group schemes.

The first serious calculation comes in Section \ref{sec:generic mackey} in establishing a \textit{generic Mackey formula} (Theorem \ref{thm:generic mackey}). In general, a Mackey formula should relate the Lusztig induction functors and their adjoints, giving a formula for the composition ${}^* R_{\bbL_r,\bbQ_r}^{\bbG_r} \circ R_{\bbM_r,\bbP_r}^{\bbG_r}$. The conjectural Mackey formula of course \textit{contains} the Scalar Product Conjecture, and even in the classical setting $r=0$, establishing the Mackey formula is well documented to be difficult: it has been resolved in many (but not all) special cases by work of \cite{DL83,DM20,BMM93,BM11,Tay18,Lus20a}. In the special case that (at least) one of $L$ or $M$ is a torus, the $r=0$ formula can be obtained by a single argument due to Deligne--Lusztig \cite{DL83} and Lusztig (see \cite{DM20}). We will prove (Theorem \ref{thm:generic mackey}) a formula for ${}^*R_{\bbT_r,\bbB_r}^{\bbG_r} \circ R_{\bbM_r,\bbP_r}^{\bbG_r}$ under a genericity condition governed by $M$; the proof combines Lusztig's $r=0$ argument together with techniques established in \cite{Lus04,Sta09,CI21-RT}. %We will leverage the settings where the Mackey formula is accessible in order to make progress on the Scalar Product Conjecture.

\begin{displaytheorem}\label{thm:display mackey}
  Let $\rho$ be any representation of $\bbM_r^\sigma$ which is $(M,G)$-generic. Then
  \begin{equation*}
    {}^*R_{\bbT_r, \bbB_r}^{\bbG_r} \circ R_{\bbM_r, \bbP_r}^{\bbG_r}(\rho)  = \sum_{w \in \bbT_r^\sigma \backslash \cS(\bbT_r, \bbM_r)^\sigma/\bbM_r^\sigma} {}^* R_{{}^w\bbT_r, {}^w\bbB_r \cap \bbM_r}^{\bbM_r}(\ad(w^{-1})^*\rho).
  \end{equation*}
\end{displaytheorem}

The overarching idea of this paper is the Scalar Product Conjecture can be resolved by iteratively leveraging generic Mackey formulas. This target iteration directs us to the problem of describing $R_{\bbT_r,\bbB_r}^{\bbG_r}(\theta)$ as the composition of parahoric Lusztig inductions. This relies on Kaletha's work \cite{Kal19} generalizing Howe's $\GL_n$ work \cite{How77} factorizing characters of maximal tori: given a character $\theta$ of $\bbT_r^\sigma$, Kaletha proves that one can write down a sequence of characters $\phi_{-1}, \phi_0, \ldots, \phi_d$ of increasing depth, where each $\phi_i$ is a character of $(\bbG_r^i)^\sigma$ for an increasing sequence of Levi subgroups $G^{-1} = T \subseteq G^0 \subset \cdots \subset G^d = G$. (The characters $\phi_i$ are not uniquely determined by $\theta$, but their depths and the associated Levi subgroups $G^i$ are.) Given a Howe factorization $\vec \phi$ of a character $\theta$ of $\bbT_r^\sigma$, we may define a virtual $\bbG_r^\sigma$-representation $r_{\bbT_r,\bbB_r}^{\bbG_r}(\vec \phi; \vec P)$ of $\bbG_r^\sigma$ obtained by inflating to a larger depth, tensoring by a generic character, and applying parahoric Lusztig induction. We prove (Proposition \ref{prop:Howe}) that if $(\theta,\bbT_r,\bbB_r)$ is split-generic, then $R_{\bbT_r,\bbB_r}^{\bbG_r}(\theta) \cong r_{\bbT_r,\bbB_r}^{\bbG_r}(\vec \phi)$. The reason this isomorphism holds is due to the following theorem:

\begin{displaytheorem}\label{thm:display depth}
  Assume $\bfT$ is elliptic. If $\theta \from \bbT_r^\sigma \to \overline \QQ_\ell^\times$ is a character which factors through $\bbT_s^\sigma$ for some $s < r$, then we have an isomorphism of virtual $\bbG_r^\sigma$-representations
  \begin{equation*}
    R_{\bbT_r,\bbB_r}^{\bbG_r}(\theta) \cong R_{\bbT_s,\bbB_s}^{\bbG_s}(\theta).
  \end{equation*}
\end{displaytheorem}

Theorem \ref{thm:display depth} follows from Theorem \ref{thm:fiber cohomology}, the true crux of this paper. Consider the $\overline \FF_q$-schemes $X_{\bbT_r,\bbB_r}^{\bbG_r}$ defining the functor $R_{\bbT_r,\bbB_r}^{\bbG_r}$. Theorem \ref{thm:fiber cohomology} calculates the cohomology of the fibers of
\begin{equation}\label{eq:r to r-1}
  X_{\bbT_r,\bbB_r}^{\bbG_r} \to X_{\bbT_{r-1},\bbB_{r-1}}^{\bbG_r}.
\end{equation}
This allows us to obtain (see  Corollary \ref{cor:infinite level}) a definition of $\ell$-adic homology groups for the infinite-depth parahoric Deligne--Lusztig variety $X_{\bbT_\infty,\bbB_\infty}^{\bbG_\infty}$. This for example endows any sufficiently well understood  $p$-adic Deligne--Lusztig space (conjectured by Lusztig \cite{Lus79} and studied by Ivanov \cite{I23arc,I23orbit} in Coxeter cases) with $\ell$-adic homology groups and shows that they encode the same representations as $X_{\bbT_\infty,\bbB_\infty}^{\bbG_\infty}$ (see Remarks \ref{rem:relation to lusztig} and \ref{rem:indirect depth compat} for more comments in this direction). Prior to this paper, Theorem \ref{thm:display depth} was known only in the setting that $G$ is an inner form of $\GL_n$ \cite{Lus79,Boy12,CI21-MA}, and in these cases, the fibers of \eqref{eq:fiber cohomology} are disjoint unions of a fixed affine space $\bbA^N$. For general $G$, while it is conceivable that this also happens, to establish Theorem \ref{thm:display depth} we prove the weaker statement that the fibers of \eqref{eq:r to r-1} share the same \textit{cohomology} as disjoint unions of $\bbA^N$ (Theorem \ref{thm:fiber cohomology}).  
%We prove (Theorem \ref{thm:fiber cohomology}) that if $\bfT$ is elliptic, then for all $i \geq 0$,
% \begin{equation*}
%   H_c^i(X_{\bbT_r,\bbB_r}^{\bbG_r}, \overline \QQ_\ell)^{\ker(\bbT_r^\sigma \to \bbT_{r-1}^\sigma)} = H_c^{i-2d}(X_{\bbT_{r-1},\bbB_{r-1}}^{\bbG_{r-1}}, \overline \QQ_\ell(d)),
% \end{equation*}
% where $d$ is the number of positive roots of $\bfT$ in $\bfG$.

After establishing Theorems \ref{thm:display mackey} and \ref{thm:display depth}, the Main Theorem is simple to prove: using the isomorphism $R_{\bbT_r,\bbB_r}^{\bbG_r}(\theta) \cong r_{\bbT_r}^{\bbG_r}(\vec \phi)$ from Proposition \ref{prop:Howe} (which depends on Theorem \ref{thm:display depth}), it is equivalent to calculate the inner product
\begin{equation*}
  \langle r_{\bbT_r}^{\bbG_r}(\vec \phi), R_{\bbT_r',\bbB_r'}^{\bbG_r}(\theta') \rangle,
\end{equation*}
which we do by applying the generic Mackey formula (Theorem \ref{thm:display mackey}) to successively peel off layers in the Howe factorization. This is the content of Section \ref{sec:scalar}.

The techniques in this paper have direct analogues for the functors arising from the Drinfeld stratification $X_{\bbT_r,\bbB_r}^{\bbL_r\bbG_r^+}$ of parahoric Deligne--Lusztig varieties (see Definition \ref{def:drinfeld}) in the sense of \cite{CI21-SM}. We explain the minor modifications required to do this and  establish the Scalar Product Conjecture for Drinfeld strata in Section \ref{sec:drinfeld}.

%\mbox{}

Allow us to mention an immediate application of the results of this paper. Assume that $T$ is elliptic, the setting of the Corollary above and of Theorem \ref{thm:display depth}. In forthcoming work with M.\ Oi, under a largeness condition on $q$, we determine $R_{\bbT_\infty,\bbB_\infty}^{\bbG_\infty}(\theta)$ in terms of Yu's construction \cite{Yu01} of tame supercuspidal representations. In particular, we can then describe $R_{\bbT_\infty,\bbB_\infty}^{\bbG_\infty}(\theta)$ in terms of Kaletha's local Langlands correspondence for regular supercuspidal representations \cite{Kal19}, thereby removing the torality assumption required in our previous work \cite{CO21}. Our methods vitally depend on the scalar product formula. 

\subsection*{Acknowledgments} The author sincerely thanks Alexander Ivanov for pointing out an error in the first version of this paper. She also thanks Eva Viehmann and her learning seminar for careful comments and suggestions. Additionally, she thanks Jessica Fintzen, Tasho Kaletha, Masao Oi, and Xinwen Zhu for discussions on an earlier draft. The author was partially supported by NSF grants DMS-2101837, DMS-2401114, and a Sloan Research Fellowship.

%In light of Theorem \ref{thm:display depth} and Corollary \ref{cor:infinite level}, this work with Oi calculates the representations $R_{\bbT_\infty,\bbB_\infty}^{\bbG_\infty}(\theta)$ within the context of the local Langlands program. and describes the representation-theoretic content encoded in Lusztig's conjectural $p$-adic Deligne--Lusztig varieties (see Remark \ref{rem:relation to lusztig} for more comments).

% In an adjacent direction, let us mention that Theorem \ref{thm:display depth} is proved by studying the geometry of the parahoric Deligne--Lusztig variety $X_r$ defining the functor $R_{\bbT_r,\bbB_r}^{\bbG_r}$. We in fact prove something much stronger---we prove that, under the assumption that $\bfT$ is elliptic, there is a canonical embedding
% \begin{equation}\label{eq:embedding}
%   H_c^i(X_{r-1}, \overline \QQ_\ell) \hookrightarrow H_c^{i+2N}(X_r, \overline \QQ_\ell(N)).
% \end{equation}
% This inclusion is induced by identifying $H_c^i(X_{r-1}, \overline \QQ_\ell)$ with the fixed-points of $H_c^i(X_r, \overline \QQ_\ell)$ under the action of $\ker(\bbT_r^\sigma \to \bbT_{r-1}^\sigma)$. If for any smooth $\overline \FF_q$-scheme $S$ of pure dimension $d$ we define $H_i(S, \overline \QQ_\ell) \colonequals H_c^{2d-i}(S, \overline \QQ_\ell(d))$, then \eqref{eq:embedding} allows us to define
% \begin{equation*}
%   H_i(X_\infty) \colonequals \varinjlim_r H_i(X_r, \overline \QQ_\ell), \qquad \text{where $X_\infty \colonequals \varprojlim_r X_r$}.
% \end{equation*}
% \textcolor{red}{relate this to Alex's work on loop DL}

\setcounter{tocdepth}{1}

\tableofcontents

\newpage

\section{Notation}\label{sec:notation}

Let $F$ be a non-archimedean local field and let $\breve F$ denote the completion of the maximal unramified extension of $F$. We write $\cO_F$ and $\breve \cO$ for the ring of integers of $F$ and $\breve F$. Write $k_F \cong \bbF_q$ and $k = \overline \bbF_q$ for the residue fields of $F$ and $\breve F$; note that $k$ is an algebraic closure of $k_F$. Choose a uniformizer $\varpi$ of $F$. For any finite group $\G$, we write $\cR(\G)$ for the representation ring of $\G$, with coefficients in $\overline \QQ_\ell$ for $\ell \neq p = \Char(k_F)$.

Let $\bfG$ be a connected reductive group over $\breve F$ and $\bfT \hookrightarrow \bfG$ a split torus. We denote by $\Phi(\bfG,\bfT)$ its corresponding root system. Choose a point $\x$ in the apartment of $\bfT$ and fix a positive integer $r > 0$. By Bruhat--Tits theory and a construction of Yu \cite{Yu15}, we have an associated smooth affine $\breve \cO$-model $\cG_{\x,r}$ of $\bfG$ such that $\cG_{\x,r}(\breve \cO)$ is the $r$th Moy--Prasad filtration subgroup \cite{MP94,MP96} of the parahoric group subgroup $\cG_{\x,0}(\breve \cO) \subset G(\breve F)$. Following \cite[Section 2.5]{CI21-RT}, we consider the perfectly of finite type smooth affine group scheme $\bbG_{s:r+}$ representing the perfection of the functor
\begin{equation}
  R \mapsto \cG_{\x,s}(\bbW(R))/\cG_{\x,r+}(\bbW(R)),
\end{equation}
where $R$ is any $k$-algebra. Here, $\bbW$ denotes the Witt ring associated to $F$ if $F$ has characteristic $0$ and $\bbW(R) = R[\![\varpi]\!]$ if $F$ has positive characteristic. As in \cite[Section 2.6]{CI21-RT}, associated to any closed subgroup scheme $\bfH$ of $\bfG$, we have an associated closed subgroup scheme $\bbH_{s:r+}$ of $\bbG_{s:r+}$. Abusing notation, we define
\begin{equation*}
  \bbG_r \colonequals \bbG_{0:r+}.
\end{equation*}
We denote by $W_{\bbG_r}(\bbT_r)$ the absolute Weyl group of $\bbG_r$. 

Throughout this paper, we assume that $\bfG$ and $\bfT$ each arise as the base-change of a connected reductive group $G$ and a torus $T$ defined over $F$. Let $E$ be the splitting field of $T$. We have associated Frobenius endomorphisms $\sigma \from \bfG \to \bfG$ and $\sigma \from \bbG_r \to \bbG_r$ stabilizing $\bfT$ and $\bbT_r$ respectively. We use the superscript ${}^\sigma$ to denote the $\sigma$-fixed points, so that for example $\bbG_r^\sigma$ is a quotient of a parahoric subgroup of $G(F)$ and $\bbT_r^\sigma$ is a subquotient of $T(F)$. If $T$ satisfies some property, we will say that $\bfT$ satisfies that property \textit{over $F$}.

\section{Parahoric Lusztig induction}\label{sec:parahoric lusztig}

\subsection{Definitions}

Completely analogously to parahoric Deligne--Lusztig induction as defined in \cite{Lus04,Sta09,CI21-RT}, we may define parahoric Lusztig induction.

\begin{definition}
  Let $\bfM$ be a $F$-rational Levi subgroup of $\bfG$ containing $\bfT$ and let $\bfP$ be a parabolic subgroup of $\bfG$ with Levi component $\bfM$. Let $\bfN$ denote the unipotent radical of $\bfP$. Define the parahoric Lusztig variety to be
  \begin{equation*}
    X_{\bbM_r, \bbP_r}^{\bbG_r} \colonequals \{x \in \bbG_r : x^{-1}\sigma(x) \in \sigma(\bbN_r)\}.
  \end{equation*}
  Note that this has a natural left action of $\bbG_r^\sigma$ and a natural right action of $\bbM_r^\sigma$ given by
  \begin{equation*}
    (g,m) \from x \mapsto gxm.
  \end{equation*}
  We point out to the reader that we may not have $\sigma(\bbN_r) = \bbN_r$. Let $n$ be a positive integer such that $\sigma^n(\bbN_r) = \bbN_r$; then $X_{\bbM_r,\bbP_r}^{\bbG_r}$ is defined over $\FF_{q^n}$.
\end{definition}

\begin{definition}\label{def:induction}
  We define the functor
  \begin{equation*}
    R_{\bbM_r, \bbP_r}^{\bbG_r} \from \cR(\bbM_r^\sigma) \to \cR(\bbG_r^\sigma)
  \end{equation*}
  by the formula
  \begin{equation*}
    R_{\bbM_r,\bbP_r}^{\bbG_r}(\chi)(g) = \frac{1}{|\bbM_r^\sigma|} \sum_{m \in \bbM_r^\sigma} \tr((g,m) ; H_c^*(X_{\bbM_r, \bbP_r}^{\bbG_r}, \overline \QQ_\ell)) \cdot \overline{\chi(m)}.
  \end{equation*}
  The adjoint functor 
  \begin{equation*}
    {}^* R_{\bbM_r,\bbP_r}^{\bbG_r} \from \cR(\bbG_r^\sigma) \to \cR(\bbM_r^\sigma)
  \end{equation*}
  is given by the formula
  \begin{equation*}
    {}^* R_{\bbM_r,\bbP_r}^{\bbG_r}(\psi)(m) = \frac{1}{|\bbG_r^\sigma|} \sum_{g \in \bbG_r^\sigma} \tr((g,m); H_c^*(X_{\bbM_r,\bbP_r}^{\bbG_r}, \overline \QQ_\ell)) \cdot \overline{\psi(g)}.
  \end{equation*}
\end{definition}

\subsection{Properties}

We present several natural properties of parahoric Lusztig functors.

\begin{proposition}[transitivity]\label{prop:transitivity}
  Let $\bfQ \subset \bfP$ be two parabolic subgroups of $\bfG$ and let $\bfL \subset \bfM$ be $F$-rational Levi subgroups of $\bfQ$ and $\bfP$ respectively. Then
  \begin{equation*}
    R_{\bbM_r, \bbP_r}^{\bbG_r} \circ R_{\bbL_r, \bbM_r \cap \bbQ_r}^{\bbM_r} = R_{\bbL_r, \bbQ_r}^{\bbG_r}.
  \end{equation*}
\end{proposition}

\begin{proof}
  We have Levi decompositions  $\bfQ = \bfL \ltimes \bfV$ and $\bfP = \bfM \ltimes \bfN$ which induces a Levi decomposition $\bfQ \cap \bfM = \bfL \ltimes (\bfV \cap \bfM)$ of the parabolic $\bfQ \cap \bfM$ in $\bfM$. We may consider the three functors $R_{\bbM_r,\bbP_r}^{\bbG_r},R_{\bbL_r, \bbQ_r \cap \bbM_r}^{\bbM_r}, R_{\bbL_r,\bbQ_r}^{\bbG_r}$. We would like to show that there is an isomorphism
  \begin{equation*}
    H_c^*(X_{\bbM_r,\bbP_r}^{\bbG_r}, \overline \QQ_\ell) \otimes_{\overline \QQ_\ell[\bbM_r^\sigma]} H_c^*(X_{\bbL_r, \bbQ_r \cap \bbM_r}^{\bbM_r}, \overline \QQ_\ell) \cong H_c^*(X_{\bbL_r, \bbQ_r}^{\bbG_r}, \overline \QQ_\ell).
  \end{equation*}
  To do this, we will prove that we have a $(\bbG_r^\sigma \times \bbL_r^\sigma)$-equivariant isomorphism of varieties
  \begin{equation*}
    X_{\bbM_r,\bbP_r}^{\bbG_r} \times_{\bbM_r^\sigma} X_{\bbL_r,\bbQ_r \cap \bbM_r}^{\bbM_r} \to X_{\bbL_r,\bbQ_r}^{\bbG_r}
  \end{equation*}
  given by restricting the multiplication map $\bbG_r \times \bbM_r \to \bbG_r$. 
  To see surjectivity, choose any $y \in X_{\bbL_r,\bbQ_r}^{\bbG_r}$ and first observe that $\sigma(\bbV_r) = (\sigma(\bbV_r) \cap \bbM_r)(\sigma(\bbV_r) \cap \sigma(\bbN_r))$. Hence we may write $y^{-1} \sigma(y) = mn$. Surjectivity of the Lang map implies that we may choose $m_0 \in \bbM_r$ such that $m_0^{-1} \sigma(m_0) = m$; note that $m_0 \in X_{\bbL_r,\bbQ_r \cap \bbM_r}^{\bbM_r}$. Then $g_0 \colonequals ym_0^{-1}$ has the property $g_0^{-1}\sigma(g_0) = m_0 y^{-1} \sigma(y) \sigma(m_0)^{-1} = \sigma(m_0) n \sigma(m_0)^{-1} \in \sigma(m_0) \sigma(\bbN_r) \sigma(m_0)^{-1} = \sigma(\bbN_r)$; in other words, $g_0 \in X_{\bbM_r,\bbP_r}^{\bbG_r}$. Hence we see that $y$ has a preimage $(g_0,m_0)$. To see injectivity, it suffices to prove that if $g,g' \in \bbG_r$ satisfy $g^{-1}\sigma(g), g'{}^{-1} \sigma(g') \in \sigma(\bbN_r)$ and $m,m' \in \bbM_r$ satisfy $m^{-1}\sigma(m), m'{}^{-1}\sigma(m') \in \sigma(\bbV_r) \cap \bbM_r$ and are such that $gm = g'm'$, then $g \in g'\bbM_r^\sigma$. Let $\gamma \colonequals g'{}^{-1} g = m^{-1} m' \in \bbM_r$. Then setting $g^{-1}\sigma(g) = n$ and $g'{}^{-1}\sigma(g') = n'$, we have $\sigma(\gamma) = \sigma(g'{}^{-1}g) = (g'n')^{-1}(gn) = n'{}^{-1} \gamma n$. Therefore $\gamma^{-1} \sigma(\gamma) = \gamma^{-1} n'{}^{-1} \gamma n \in \sigma(\bbN_r)$ (since $\bbM_r$ normalizes $\sigma(\bbN_r)$) and therefore $\gamma^{-1} \sigma(\gamma) \in 
  \sigma(\bbN_r) \cap \bbM_r = \{1\}$. This proves that we have an isomorphism on points and the proposition follows.
\end{proof}

% If $\bfP$ is in fact $F$-rational, then the functor $R_{\bbM_r,\bbP_r}^{\bbG_r}$ is simply given by parabolic induction: 
% \begin{equation*}
%   R_{\bbM_r,\bbP_r}^{\bbG_r} = \Ind_{\bbP_r^\sigma}^{\bbG_r^\sigma} \circ \Inf_{\bbM_r^\sigma}^{\bbP_r^\sigma}.
% \end{equation*}
% This follows from the surjectivity of the Lang map on the $\sigma$-stable $\bbN_r$. The general version of this tells us that $R_{\bbM_r,\bbP_r}^{\bbG_r}$ factors into a composition of functors, the second of which is given by parabolic induction. \textcolor{red}{Better: Write Lemma \ref{lem:pInd factor} as a special case of Proposition \ref{prop:transitivity}.}

Proposition \ref{prop:transitivity} has the following special case:

\begin{lemma}\label{lem:pInd factor}
  Let $\bfV$ be the largest $\sigma$-stable subgroup scheme of $\bfN$; then  $\bfV$ is the unipotent radical of a $F$-rational parabolic $\bfQ$ with Levi component $\bfL$. Then
  \begin{equation*}
    R_{\bbM_r,\bbP_r}^{\bbG_r} = \Ind_{\bbQ_r^\sigma}^{\bbG_r^\sigma} \circ \Inf_{\bbL_r^\sigma}^{\bbQ_r^\sigma} \circ R_{\bbM_r, \bbL_r \cap \bbP_r}^{\bbL_r}.
  \end{equation*}
\end{lemma}

\begin{proof}
  If $\bfP$ is $F$-rational, then $R_{\bbM_r,\bbP_r}^{\bbG_r}$ is simply given by parabolic induction:
  \begin{equation}\label{eq:pInd}
    R_{\bbM_r,\bbP_r}^{\bbG_r} = \Ind_{\bbP_r^\sigma}^{\bbG_r^\sigma} \circ \Inf_{\bbM_r^\sigma}^{\bbP_r^\sigma}.
  \end{equation}
  This follows from the surjectivity of the Lang map on $\bbN_r$: if $g \in \bbG_r$ is such that $g^{-1} \sigma(g) \in \bbN_r$, then there exists an $n \in \bbN_r$ such that $n^{-1} \sigma(n) = g^{-1} \sigma(g)$ so that we have 
  \begin{equation*}
    X_{\bbM_r,\bbP_r}^{\bbG_r}/\bbN_r \cong \bbG_r^\sigma/\bbN_r^\sigma.
  \end{equation*}
  Now the lemma follows from \eqref{eq:pInd} together with Proposition \ref{prop:transitivity}.
\end{proof}

% \begin{proof}
%   This follows from obtaining a $(\bbM_r^\sigma \times \bbG_r^\sigma)$-equivariant isomorphism
%   \begin{equation}\label{eq:isomorphism}
%     \bigsqcup_{\gamma \in \bbG_r^\sigma/\bbQ_r^\sigma} \gamma \cdot X_{\bbM_r,\bbL_r \cap \bbP_r}^{\bbL_r} \cong X_{\bbM_r,\bbP_r}^{\bbG_r}/\bbV_r.
%   \end{equation}
%   Consider natural map
%   \begin{equation*}
%     \bigcup_{\gamma \in \bbG_r^\sigma} \gamma \cdot X_{\bbM_r,\bbL_r \cap \bbP_r}^{\bbL_r} \to X_{\bbM_r,\bbP_r}^{\bbG_r}/\bbV_r
%   \end{equation*}
%   induced by multiplication. 
%   We check surjectivity: Let $g \in \bbG_r$ be such that $g^{-1} \sigma(g) \in \bbN_r$. By the surjectivity of the Lang map on $\bbV_r$, we may modify $g$ by $\bbV_r$ so that $g^{-1} \sigma(g) \in \bbL_r \cap \bbN_r$. By surjectivity of the Lang map on $\bbL_r$, we can find $l \in \bbL_r$ such that $l^{-1} \sigma(l) = n_2$ so that now $g^{-1} \sigma(g) = l^{-1} \sigma(l)$ and $\gamma \colonequals gl^{-1} \in \bbG_r^\sigma$. This proves surjectivity. We calculate the fibers of this map: if $\gamma,\gamma' \in \bbG_r^\sigma$ and $l,l' \in X_{\bbM_r,\bbL_r \cap \bbP_r}^{\bbL_r}$ are such that $\gamma l = \gamma' l' v$ for some $v \in \bbV_r$, then $\gamma = \gamma' \cdot l' l^{-1} \cdot lvl^{-1}$ and necessarily $l'l^{-1} \in \bbL_r^\sigma$ and $lvl^{-1} \in \bbV_r^\sigma$. Since $v$ can be chosen arbitrarily as can $l' l^{-1}$, we see that the above map descends to the desired isomorphism in \eqref{eq:isomorphism}.
% \end{proof} 

\begin{lemma}\label{lem:invariants}
  For any $s \leq r$, 
  we have a commutative diagram
  \begin{equation*}
    \begin{tikzcd}
      \cR(\bbM_r^\sigma) \ar{rr}{R_{\bbM_r,\bbP_r}^{\bbG_r}} \ar{d}[left]{(-)^{\bbM_{s+:r+}^\sigma}} && \cR(\bbG_r^\sigma) \ar{d}{(-)^{\bbG_{s+:r+}^\sigma}}\\
      \cR(\bbM_{s}^\sigma) \ar{rr}{R_{\bbM_{s},\bbP_{s}}^{\bbG_{s}}} && \cR(\bbG_{s}^\sigma)
    \end{tikzcd}
  \end{equation*}
  where the vertical arrows are given by taking invariants.
  % \begin{equation*}
  %   F_{s+:r+} \circ R_{\bbM_r,\bbP_r}^{\bbG_r} = R_{\bbM_{s},\bbP_{s}}^{\bbG_{s}},
  % \end{equation*}
  % where $F_{s+:r+} \from \cR(\bbG_r^\sigma) \to \cR(\bbG_{s}^\sigma)$ denotes the functor obtained by taking $\bbG_{s+:r+}^\sigma$-invariants. 
  %we have $H_c^*(X_r, \overline \QQ_\ell)^{G_{s+:r+}} \cong H_c^*(X_{s},\overline\QQ_\ell)$ as representations of $\bbG_s^\sigma \times \bbM_{s}^\sigma$. %In particular, if $\rho$ is a representation of $\bbG_s^\sigma$ of depth $r > s$, then $R_{\bbG_s \subset \bbP_r}^{\bbG_r}(\rho)^{G_{s+:r+}} = 0.$
\end{lemma}

\begin{proof}
  Consider the surjective map $X_{\bbM_r,\bbP_r}^{\bbG_r} \to X_{\bbM_s,\bbP_s}^{\bbG_s}$. For any $\bar g \in X_{\bbM_s,\bbP_s}^{\bbG_s}$, choose a lift $g \in X_{\bbM_r,\bbP_r}^{\bbG_r}$ and write $\sigma(u) = g^{-1} \sigma(g)$. Then the fiber over $\bar g$ is isomorphic to
  \begin{equation*}
    \{g_r \in \bbG_{s+:r+} : (g_r g)^{-1}\sigma(g_r g) \in \sigma(\bbN_r)\} 
    = \{g_r \in \bbG_{s+:r+} : g_r^{-1}\sigma(g_r) \in g \sigma(\bbN_r) g^{-1}\},
  \end{equation*}
  where the equality holds since $g \sigma(\bbN_r) \sigma(g)^{-1} = g \sigma(\bbN_r) \sigma(u)^{-1} g^{-1} = g \sigma(\bbN_r) g^{-1}$. Hence the fibers of $\bbG_{s+:r+}^\sigma\backslash X_{\bbM_r,\bbP_r}^{\bbG_r} \to X_{\bbM_s,\bbP_s}^{\bbG_s}$ are each isomorphic to $\sigma(\bbN_{s+:r+}),$ an affine space. Therefore we see that as virtual representations of $\bbG_r^\sigma \times \bbM_r^\sigma$,
  \begin{equation*}
    H_c^*(X_{\bbM_r,\bbP_r}^{\bbG_r}, \overline \QQ_\ell)^{\bbG_{s+:r+}^\sigma} \cong H_c^*(X_{\bbM_s,\bbP_s}^{\bbG_s}, \overline \QQ_\ell). \qedhere
  \end{equation*}
  %But now this implies that $\bbM_{s+:r+}^\sigma$ must act trivially on $H_c^*(X_r, \overline \QQ_\ell)^{\bbG_{s+:r+}^\sigma}$, so in particular, we see that if $\rho^{\bbM_{s+:r+}^\sigma} = 0$, then $R_{\bbM_r,\bbP_r}^{\bbG_r}(\rho)^{\bbG_{s+:r+}^\sigma} = 0$.
\end{proof}

\begin{lemma}\label{lem:der}
  Let $\bfG_\der$ be the derived subgroup of $\bfG$ and consider the associated subgroups $\bbG_r^{\der}$, $\bbM_r^{\der}$, $\bbP_r^{\der},$ and $\bbN_r^{\der} = \bbN_r$. We have an isomorphism
  \begin{equation*}
    \bigsqcup_{\tau \in \bbT_r^\sigma/(\bbT_r^{\der})^\sigma} X_{\bbM_r^{\der}, \bbP_r^\der}^{\bbG_r^\der} \cdot \tau \to X_{\bbM_r,\bbP_r}^{\bbG_r}.
  \end{equation*}
\end{lemma}

\begin{proof}
  Since $\bbN_r^{\der} = \bbN_r$, it is clear that the map in question is an inclusion. It remains to show surjectivity. If $x \in X_{\bbM_r,\bbP_r}^{\bbG_r}$, then by definition we have $x^{-1} \sigma(x) \in \sigma(\bbN_r) = \sigma(\bbN_r^{\der})$, and so in particular $x \bbG_r^\der = \sigma(x) \bbG_r^{\der} \in (\bbG_r/\bbG_r^{\der})^\sigma$. By \cite[Lemmas 3.1.3, 3.1.4]{Kal19}, we have $(\bbG_r/\bbG_r^{\der})^\sigma = \bbT_r^\sigma/(\bbT_r^\der)^\sigma$, which therefore implies that $x = y \cdot \tau$ for some $y \in \bbG_r^\der$ and $\tau \in \bbS_r^\sigma$. To conclude, we note that $y^{-1} \sigma(y) = \tau x^{-1} \sigma(x) \tau^{-1} \in \tau \sigma(\bbN_r) \tau^{-1} = \sigma(\bbN_r^\der)$.
\end{proof}

% \begin{lemma}\label{lem:der}
%   Let $\bfG^\der$ be the derived subgroup of $\bfG$ and consider the associated subgroups $\bbG_r^{\der}$, $\bbT_r^{\der}$, $\bbB_r^{\der},$ and $\bbU_r^{\der} = \bbU_r$. We have an isomorphism
%   \begin{equation*}
%     \bigsqcup_{\tau \in \bbT_r^\sigma/(\bbT_r^{\der})^\sigma} X_{\bbT_r^{\der}, \bbB_r^\der}^{\bbG_r^\der} \cdot \tau \to X_{\bbT_r,\bbB_r}^{\bbG_r}.
%   \end{equation*}
% \end{lemma}

% \begin{proof}
%   Since $\bbU_r^{\der} = \bbU_r$, it is clear that the map in question is an inclusion. It remains to show surjectivity. If $x \in X_{\bbT_r,\bbB_r}^{\bbG_r}$, then by definition we have $x^{-1} \sigma(x) \in \sigma(\bbU_r) = \sigma(\bbU_r^{\der})$, and so in particular $x \bbG_r^\der = \sigma(x) \bbG_r^{\der} \in (\bbG_r/\bbG_r^{\der})^\sigma$. By \cite[Lemma 3.1.3]{Kal19}, we have $(\bbG_r/\bbG_r^{\der})^\sigma = \bbS_r^\sigma/(\bbS_r^\der)^\sigma$, which therefore implies that $x = y \cdot \tau$ for some $y \in \bbG_r^\der$ and $\tau \in \bbS_r^\sigma$. To conclude, we note that $y^{-1} \sigma(y) = \tau x^{-1} \sigma(x) \tau^{-1} \in \tau \sigma(\bbU_r) \tau^{-1} = \sigma(\bbU_r^\der)$.
% \end{proof}

%\subsection{Twisting}\label{subsec:twisting}

We now use Lemma \ref{lem:der} to establish the behavior of the functor $R_{\bbM_r,\bbP_r}^{\bbG_r}$ under twisting.

\begin{proposition}\label{prop:twisting}
  Let $\tilde \phi \from \bbG_r^\sigma \to \overline \QQ_\ell^\times$ be any character and write $\phi = \tilde \phi|_{\bbT_r^\sigma}$. Assume that $\tilde \phi|_{(\bbG_r^\der)^\sigma} \equiv 1$. Then for any $\chi \in \cR(\bbM_r^\sigma)$,
  \begin{equation*}
    R_{\bbM_r,\bbP_r}^{\bbG_r}(\chi \otimes \phi) \cong R_{\bbM_r,\bbP_r}^{\bbG_r}(\chi) \otimes \tilde \phi.
  \end{equation*}
\end{proposition}

\begin{proof}
  We use the Deligne--Lusztig fixed-point formula \cite[Theorem 3.2]{DL76} and Lemma \ref{lem:der}. By definition,
  \begin{equation}\label{eq:RTG def}
    R_{\bbM_r,\bbP_r}^{\bbG_r}(\chi \otimes \phi)(g)
    = \frac{1}{|\bbM_r^\sigma|} \sum_{m \in \bbM_r^\sigma} \tr((g,m); H_c^*(X_{\bbM_r,\bbP_r}^{\bbG_r},\overline \QQ_\ell)) \cdot \overline{\chi(m)} \cdot \phi(m)^{-1}.
  \end{equation}
  Let us work with the summand corresponding to $m$. Under the isomorphism in Lemma \ref{lem:der}, the action of $(g,m)$ on $x \cdot \tau$ for $x \in X_{\bbM_r^\der,\bbP_r^\der}^{\bbG_r^\der}$ is given by
  \begin{equation}\label{eq:der action}
    (g,m) \cdot (x \cdot \tau) = g_0 (\tau_g x \tau_g^{-1}) \tau_g \tau m_0 \tau^{-1} \tau_g^{-1} \cdot \tau_g \tau \tau_m,
  \end{equation}
  where $g_0 \in (\bbG_r^{\der})^\sigma$ and $\tau_g \in \bbT_r^\sigma$ are such that $g = g_0 \tau_g$ and $m_0 \in (\bbM_r^{\der})^\sigma$ and $\tau_m \in \bbT_r^\sigma$ are such that $m = m_0 \tau_m$. We see that if $\tau_g \tau_m \neq 1$, then $(g,m)$ freely permutes the copies of $X_{\bbM_r^\der,\bbP_r^\der}^{\bbG_r^\der}$. Therefore
  \begin{equation*}
    \tr((g,m); H_c^*(X_{\bbM_r,\bbP_r}^{\bbG_r}, \overline \QQ_\ell)) = 0 \qquad \text{if $\tau_m \neq \tau_g^{-1}$.}
  \end{equation*}
  Therefore, in \eqref{eq:der action}, we need only sum over $m \in \bbM_r^\sigma$ for which $\tau_m = \tau_g^{-1}$. It is at this point that we use the assumption that $\tilde \phi|_{(\bbG_r^\der)^\sigma} \equiv 1$: we then have
  \begin{equation*}
    \phi(m)^{-1} = \tilde \phi(m)^{-1} = \tilde \phi(m_0)^{-1} \tilde \phi(\tau_m)^{-1} = \tilde \phi(\tau_m)^{-1} = \tilde \phi(\tau_g) = \tilde \phi(g_0) \tilde \phi(\tau_g) = \tilde \phi(g).
  \end{equation*}
  It follows then that the summand of \eqref{eq:der action} corresponding to $m \in \bbM_r^\sigma$ is
  \begin{equation*}
    \tr((g,m); H_c^*(X_{\bbM_r,\bbP_r}^{\bbG_r}, \overline \QQ_\ell)) \cdot \overline{\chi(m)} \cdot \tilde\phi(g).
  \end{equation*}
  This implies that we may factor out $\tilde\phi(g)$ in \eqref{eq:RTG def}, and the lemma follows.
\end{proof}

\begin{remark}\label{rem:small q}
  The assumption that $\tilde \phi|_{(\bbG_r^\der)^\sigma} \equiv 1$ is a genuine assumption: the one-dimensional representations of $\bbG_r^\sigma$ are in bijection with the one-dimensional representations of $\bbG_r^\sigma/[\bbG_r^\sigma,\bbG_r^\sigma]$, so when $[\bbG_r^\sigma,\bbG_r^\sigma]$ is a proper subgroup of $(\bbG_r^\der)^\sigma$, then there obviously exist $\tilde \phi$'s which do not factor through $(\bbG_r/\bbG_r^\der)^\sigma$. We thank Masao Oi for pointing out this subtlety.
\end{remark}

\begin{remark}\label{rem:green insufficient}
  In the $r=0$ case, Proposition \ref{prop:twisting} follows from the Deligne--Lusztig character formula expressing $R_{\bbT_0,\bbB_0}^{\bbG_0}(\theta)$ in terms of $\theta$ and a Green function (which does not depend on $\theta$) \cite[Theorem 4.2]{DL76}. There is an analogous formula in the $r>0$ case, proved by exactly the same method as in \textit{op.\ cit.} However, the ``Green function'' that arises depends on $\theta|_{\bbT_{0+:r+}^\sigma}$, which makes this approach insufficient to prove Proposition \ref{prop:twisting}.
\end{remark}

\section{Generic Mackey formula for a torus}\label{sec:generic mackey}

\subsection{Generic characters and Howe factorizations}\label{subsec:generic characters}

Let $\bfH$ be a connected reductive subgroup of $\bfG$ containing $\bfT$.

\begin{definition}[weak $(\bfH,\bfG)$-genericity]
  A character $\phi$ of $\bfH(F)$ is \textit{weakly $(\bfH,\bfG)$-generic of depth $r$} if $\phi$ has depth $r$ (i.e.\ $\phi|_{\bfH(F)_{x,r+}} = \triv$) and for all $\alpha \in \Phi(\bfG,\bfT) \smallsetminus \Phi(\bfH,\bfT)$, we have $\phi|_{N_{E/F}(\alpha^\vee(E_r^\times))} \neq \triv$, where $E$ is a splitting field of $\bfT$. We say a representation $\rho$ of $\bbH_r^\sigma$ is $(\bfH,\bfG)$-generic if the restriction $\rho|_{\bbH_{r:r+}^\sigma}$ is the restriction of a sum of $(\bfH,\bfG)$-generic characters of depth $r$.
\end{definition}

By \cite[Lemma 3.6.8]{Kal19}, this exactly means that $\phi$ satisfies GE1 of \cite[\S8]{Yu01}. We call this notion of genericity weak in order to distinguish it from the standard notion of genericity, which additionally requires condition GE2 of \textit{op.\ cit.} This distinction only affects finitely many primes $p$ as GE2 is automatic if $p$ is not bad for $\bfG$ and does not divide the order of $|\pi_1(\widehat \bfG_{\der})|$ (see \cite[Remark 3.4]{CO21} and \cite[\S4]{Kaletha} for more details). %As a consequence of our genericity definition not requiring GE2, Kaletha's result (see Theorem \ref{thm:Howe} below) on the existence of Howe factorizations for characters of maximal tori holds for all $p$.
Note that weakly $(\bfT,\bfG)$-generic characters of depth $r$ are exactly the regular characters in the sense of \cite[1.5]{Lus04}.

\begin{definition}[Howe factorization]\label{def:Howe}
  Set $\bfG^{-1} = \bfT$. 
  A \textit{Howe factorization} of $(\theta,\bfT)$ is a sequence of characters $\phi_i \from \bfG^i(F) \to \bbC^\times$ for $i = -1,0, \ldots, d$ with the following properties:
  \begin{enumerate}
    \setcounter{enumi}{-1}
    \item $\bfG^i$ is a twisted Levi subgroup of $\bfG$
    \item $\theta = \prod_{i=-1}^d \phi_i|_{T(F)}$.
    \item For all $0 \leq i \leq d$, the character $\phi_i$ is trivial on $\bfG_\der^i(F)$.
    \item For all $0 \leq i < d$, the character $\phi_i$ has depth $r_i$ and is weakly $(\bfG^i,\bfG^{i+1})$-generic. For $i = d$, we take $\phi_d = 1$ if $r_d = r_{d-1}$ and has depth $r_d$ otherwise. For $i = -1,$ the character $\phi_{-1}$ is trivial if $\bfG^0 = \bfT$ and otherwise satisfies $\phi_{-1}|_{T(F)_{0+}} = 1$.
  \end{enumerate}
  We call $d$ the \textit{Howe factorization length} of $(\theta,\bfT)$.
\end{definition}

Note that Howe factorizations may not be unique: there may be many choices of $\phi_i$'s which work. However, the reductive subgroups $\bfG^i$ of $\bfG$ are uniquely determined: for each positive real number $s$, consider the set of roots
\begin{equation*}
  \Phi_s \colonequals \{\alpha \in \Phi(\bfG,\bfT) : \theta|_{N_{E/F}(\alpha^\vee(E_s^\times))} = 1\}.
\end{equation*}
Then the depths $r_i$ in any Howe factorization of $(\theta,\bfT)$ are exactly the positive numbers (in fact, integers!) where $\Phi_{r_i} \neq \Phi_{r_i+\epsilon}$ for any $\epsilon > 0$, and $\bfG^i$ is by definition the connected reductive subgroup of $\bfG$ with maximal torus $\bfT$ and root system $\Phi_{r_i}$. 

%In general, it may not be the case that $\bfG^i$ is a Levi subgroup of $\bfG$. However, by \cite[Lemma 3.6.1]{Kal19}, this is automatic if $p$ is not a torsion prime for $\Phi(\bfG,\bfT)$. 

\begin{theorem}[{\cite[Lemma 3.6.1, Proposition 3.6.7]{Kal19}}]\label{thm:Howe}
  If the $\Phi_s$ associated to $(\theta,\bfT)$ are each Levi subsystems of $\Phi(\bfG,\bfT)$, then $(\theta,\bfT)$ has a Howe factorization. If $p$ is not a torsion prime for $\Phi(\bfG,\bfT)$, then any character $\theta$ of $T(F)$ has a Howe factorization.
\end{theorem}

When the $\bfG^i$'s are Levi subgroups of $\bfG$, we can make the following definition:

\begin{definition}\label{def:rTG}
  Given a Howe factorization $\vec \phi = (\phi_{-1},\ldots, \phi_d)$ of $(\theta,\bfT)$, choose a nested sequence of parabolic subgroups $\bfP^{i-1} \subset \bfG^{i}$ with Levi component $\bfG^{i-1}$ so that we have
  \begin{equation*}
    \begin{tikzcd}
      \bfT = \bfG^{-1} \ar[phantom]{r}{\subseteq} \ar[phantom]{d}[sloped]{\subsetneq} & \bfG^0 \ar[phantom]{r}{\subsetneq}\ar[phantom]{d}[sloped]{\subsetneq} & \bfG^1 \ar[phantom]{r}{\subsetneq}\ar[phantom]{d}[sloped]{\subsetneq} & \cdots \ar[phantom]{r}{\subsetneq} & \bfG^{d-1} \ar[phantom]{r}{\subsetneq}\ar[phantom]{d}[sloped]{\subsetneq} & \bfG^d = \bfG \ar[phantom]{d}[sloped]{=} \\
      \bfB = \bfP^{-1} \ar[phantom]{r}{\subseteq} & \bfP^0 \ar[phantom]{r}{\subsetneq} & \bfP^1 \ar[phantom]{r}{\subsetneq} & \cdots \ar[phantom]{r}{\subsetneq} & \bfP^{d-1} \ar[phantom]{r}{\subsetneq} & \bfG
    \end{tikzcd} 
  \end{equation*}
  Define for $0 \leq i \leq d$:
  \begin{equation*}
    r_{\bbT_{r_i}}^{\bbG_{r_i}^{i}}(\phi_{-1}, \ldots, \phi_i; \vec \bfP) = \Inf_{\bbG_{r_{i-1}}^{i \sigma}}^{\bbG_{r_i}^{i \sigma}}\left(R_{\bbG_{r_{i-1}}^{i-1},\bbP_{r_{i-1}}^{i-1}}^{\bbG_{r_{i-1}}^{i}}\left(r_{\bbT_{r_{i-1}}}^{\bbG_{r_{i-1}}^{i-1}}(\phi_{-1}, \ldots, \phi_{i-1}; \vec \bfP)\right)\right) \otimes \phi_i.
  \end{equation*}
  We write
  \begin{equation*}
    r_{\bbT_r}^{\bbG_r}(\vec \phi; \vec \bfP) \colonequals r_{\bbT_{r_d}}^{\bbG_{r_d}^{d}}(\phi_{-1}, \ldots, \phi_d; \vec \bfP).
  \end{equation*}
\end{definition}

\subsection{Generic Mackey formula}\label{subsec:generic mackey}

Set $\cS(\bbT_r, \bbM_r) = \{x \in \bbG_r(\overline \FF_q) : x^{-1} \bbT_r x \subset \bbM_r\}$. We have an identification $\bbT_r(\overline \FF_q) \backslash \cS(\bbT_r,\bbM_r) / \bbT_r(\overline \FF_q) \cong \bbT_0(\overline \FF_q) \backslash \cS(\bbT_0,\bbM_0) / \bbT_0(\overline \FF_q)$ and the generalized Bruhat decomposition
\begin{equation*}
  \bbG_0 = \bigsqcup_{w \in \bbT_0(\overline \FF_q) \backslash \cS(\bbT_0, \bbM_0)/\bbM_0(\overline \FF_q)} \bbU_0 \dot w \bbM_0 \bbN_0
\end{equation*}
pulls back to a decomposition
\begin{equation*}
  \bbG_r = \bigsqcup_{w \in \bbT_0(\overline \FF_q) \backslash \cS(\bbT_r, \bbM_r)/\bbM_0(\overline \FF_q)} \bbG_{r,w},
\end{equation*}
where
\begin{equation*}
  \bbG_{r,w} \colonequals \bbU_r \dot w \bbM_r \bbN_r = \bbB_r \bbK_{w,0+:r+} \dot w \bbP_r, \qquad \bfK_w \colonequals \bfU^- \cap \dot w \bfN^- \dot w^{-1}.
\end{equation*}

The main theorem of this section will be a formula relating the parahoric Lusztig and Deligne--Lusztig inductions $R_{\bbM_r \subset \bbP_r}^{\bbG_r}$ and $R_{\bbT_r \subset \bbB_r}^{\bbG_r}$. 

\begin{theorem}[Generic Mackey formula]\label{thm:generic mackey}
  Let $\rho$ be any representation of $\bbM_r^\sigma$ which is weakly $(\bfM,\bfG)$-generic. Then
  \begin{equation*}
    {}^*R_{\bbT_r, \bbB_r}^{\bbG_r} \circ R_{\bbM_r, \bbP_r}^{\bbG_r}(\rho)  = \sum_{w \in \bbT_r^\sigma \backslash \cS(\bbT_r, \bbM_r)^\sigma/\bbM_r^\sigma} {}^* R_{{}^w\bbT_r, {}^w\bbB_r \cap \bbM_r}^{\bbM_r}(\ad(w^{-1})^*\rho).
  \end{equation*}
  % In particular, for any character $\theta \from \bbT_r^\sigma$,
  % \begin{equation*}
  %   \langle R_{\bbM_r, \bbP_r}^{\bbG_r}(\rho), R_{\bbT_r,\bbB_r}^{\bbG_r}(\theta) \rangle_{\bbG_r^\sigma} = \sum_{w \in \bbT_r^\sigma \backslash \cS(\bbT_r, \bbM_r)/\bbM_r^\sigma} \langle {}^* R_{\bbT_r \cap {}^w \bbM_r, \bbB_r \cap {}^w \bbM_r}^{{}^w \bbM_r}(\ad(w)(\rho)), \theta \rangle_{\bbT_r^\sigma}.
  % \end{equation*}  
\end{theorem}

In Section \ref{sec:scalar}, we will apply the following reformulation of Theorem \ref{thm:generic mackey}:

\begin{corollary}\label{cor:generic mackey}
  Let $\rho$ be a weakly $(\bfM,\bfG)$-generic representation of $\bbM_r^\sigma$ and let $\theta$ be any character of $\bbT_r^\sigma$. Then
  \begin{equation*}
    \langle R_{\bbT_r,\bbB_r}^{\bbG_r}(\theta), R_{\bbM_r,\bbP_r}^{\bbG_r}(\rho) \rangle_{\bbG_r^\sigma} = \sum_{w \in \bbT_r \backslash \cS(\bbT_r, \bbM_r)^\sigma/\bbM_r} \langle R_{\bbT_r,\bbB_r \cap {}^w \bbM_r}^{{}^w \bbM_r}(\theta), \ad(w^{-1})^* \rho \rangle_{{}^w \bbM_r^\sigma}.
  \end{equation*}
\end{corollary}

\begin{proof}
  Hence for any weakly $(\bfM,\bfG)$-generic representation $\rho$ of $\bbM_r^\sigma$ and any character $\theta$ of $\bbT_r^\sigma$, we have
  \begin{align*}
    \langle R_{\bbT_r,\bbB_r}^{\bbG_r}(\theta), R_{\bbM_r,\bbP_r}^{\bbG_r}(\rho) \rangle_{\bbG_r^\sigma} 
    &= \langle \theta, {}^* R_{\bbT_r,\bbB_r}^{\bbG_r}(R_{\bbM_r,\bbP_r}^{\bbG_r}(\rho)) \rangle_{\bbT_r^\sigma} \\
    &= \sum_{w \in \bbT_r \backslash \cS(\bbT_r, \bbM_r)^\sigma/\bbM_r}\langle \theta, {}^* R_{\bbT_r,\bbB_r \cap {}^w \bbM_r}^{{}^w \bbM_r}(\ad(w^{-1})^*(\rho)) \rangle_{\bbT_r^\sigma} \\
    &= \sum_{w \in \bbT_r \backslash \cS(\bbT_r, \bbM_r)^\sigma/\bbM_r} \langle R_{\bbT_r, \bbB_r \cap {}^w \bbM_r}^{{}^w \bbM_r}(\theta), \ad(w^{-1})^* \rho \rangle_{{}^w \bbM_r^\sigma},
  \end{align*}
  where the first and third equalities hold by adjointness and the second equality holds by Theorem \ref{thm:generic mackey}.
\end{proof}

We will prove Theorem \ref{thm:generic mackey} over the course of the next three subsections, culminating with Section \ref{subsec:generic mackey proof}. The calculation proceeds by analyzing the cohomology of the fiber product $X_{\bbT_r, \bbB_r}^{\bbG_r} \times_{\bbG_r^\sigma} X_{\bbM_r, \bbP_r}^{\bbG_r}$. We have an isomorphism
\begin{align*}
  X_{\bbT_r, \bbB_r}^{\bbG_r} \times_{\bbG_r^\sigma} X_{\bbM_r, \bbP_r}^{\bbG_r} &\to \{(x,x',y) \in \sigma(\bbU_r) \times \sigma(\bbN_r) \times \bbG_r : x \sigma(y) = yx'\} \equalscolon \Sigma, \\
  (g,g') &\mapsto (g^{-1}\sigma(g), g'{}^{-1}\sigma(g'), g^{-1} g'), 
\end{align*}
where $\bfU$ is the unipotent radical of $\bfB$ and $\bfN$ is the unipotent radical of $\bfP$. Note that this isomorphism is $(\bbT_r^\sigma \times \bbM_r^\sigma)$-equivariant with respect to the action on $\Sigma$ given by
\begin{equation*}
  (t,m) \from (x,x',y) \mapsto (txt^{-1}, mx'm^{-1}, tym^{-1}).
\end{equation*}
For each double coset $w \in \bbT_r(\overline \FF_q) \backslash \cS(\bbT_r, \bbM_r)/\bbM_r(\overline \FF_q)$, set
\begin{equation*}
  \Sigma_w \colonequals \{(x,x',y) \in \Sigma : y \in \bbG_{r,w}\}.
\end{equation*}
It is clear that each $\Sigma_w$ is $(\bbT_r^\sigma \times \bbM_r^\sigma)$-stable.

\begin{lemma}\label{lem:Sigma hat}
  The cohomology of
  \begin{equation*}
    \widehat \Sigma_w \colonequals \{(x,x',u,u',z,\mu) \in \sigma(\bbU_r) \times \sigma(\bbN_r) \times \bbU_r \times \bbN_r \times \bbK_{w,0+:r+} \times \bbM_r : x \sigma(z \dot w \mu) = u z \dot w \mu u' x'\}
  \end{equation*}
  is isomorphic as a $(\bbT_r^\sigma \times \bbM_r^\sigma)$-module to that of $\Sigma_w$. This isomorphism is induced by the affine fibration $\widehat \Sigma_w \to \Sigma_w$ given by composing the isomorphism
  \begin{equation*}
    (x,x',u,u',z,\mu) \mapsto (x\sigma(u)^{-1}, x'\sigma(u'), u,u',z,\mu)   
  \end{equation*}
  with the affine fibration
  \begin{equation*}
    (x,x',u,u',z,\mu) \mapsto (x,x',uz \dot w \mu u').
  \end{equation*}
  Both these maps are $(\bbT_r^\sigma \times \bbM_r^\sigma)$-equivariant, where the action on $\widehat \Sigma_w$ is given by
  \begin{equation}\label{eq:action}
    (m,t) \from (x,x',u,u',z,\mu) \mapsto (txt^{-1}, mx'm^{-1}, tut^{-1}, mu'm^{-1}, tzt^{-1}, \dot w^{-1} t \dot w \mu m^{-1}).
  \end{equation}
  % \begin{equation}\label{eq:action}
  %   (m,t) \from (x,x',u,u',z,\mu) \mapsto (txt^{-1}, mx'm^{-1}, tut^{-1}, mu'm^{-1}, tzt^{-1}, t\mu \dot w m^{-1} \dot w^{-1}).
  % \end{equation}
\end{lemma}

Define
\begin{align*}
  \widehat \Sigma_w' &\colonequals \{(x,x',u,u',z,\mu) \in \widehat \Sigma_w : z \neq 1\}, \\
  \widehat \Sigma_w'' &\colonequals \{(x,x',u,u',z,\mu) \in \widehat \Sigma_w : z = 1\}.
\end{align*}
Theorem \ref{thm:generic mackey} will follow as a corollary (see Section \ref{subsec:generic mackey proof}) after we show that the cohomology of $\widehat \Sigma_w'$ does not contribute to the generic Mackey formula (Proposition \ref{prop:Sigma'}, proved in Section \ref{subsec:Sigma' proof}) and the cohomology of $\widehat \Sigma_w''$ is equal to the $w$-summand on the right-hand side of the Mackey formula (Proposition \ref{prop:Sigma''}, proved in Section \ref{subsec:Sigma'' proof}).

\begin{proposition}\label{prop:Sigma'}
  Let $\psi \from \bbM_{r:r+}^\sigma \to \overline \QQ_\ell^\times$ be weakly $(\bfM,\bfG)$-generic. If $w$ has a representative in $\cS(\bbM_r, \bbT_r)^\sigma$ and for all $i \geq 0$,
  \begin{equation*}
    H_c^i(\widehat \Sigma_w', \overline \QQ_\ell)_{(\psi)} = 0,
  \end{equation*}
  where $H_c^i(\widehat \Sigma_w', \overline \QQ_\ell)_{(\psi)}$ is the subspace on which $\bbM_{r:r+}^\sigma$ acts by $\psi$.
\end{proposition}

\begin{proposition}\label{prop:Sigma''}
  If $w$ has a representative in $\cS(\bbT_r, \bbM_r)^\sigma$, then we have isomorphisms of virtual $(\bbT_r^\sigma \times \bbM_r^\sigma)$-representations
  \begin{equation*}
    \sum_{i \geq 0} (-1)^i H_c^i(\widehat \Sigma_w'', \overline \QQ_\ell) \cong \sum_{i \geq 0} (-1)^i H_c^i(X_{{}^w \bbT_r, {}^w \bbB_r \cap \bbM_r}^{\bbM_r}, \overline \QQ_\ell).
    %\sum_{i \geq 0} (-1)^i H_c^{i}(X_{\bbT_r \cap {}^w \bbM_r, \bbB_r \cap {}^w \bbM_r}^{{}^w \bbM_r}, \overline \QQ_\ell),
  \end{equation*}
  where $\bbM_r^\sigma$ acts on $X_{\bbT_r, \bbB_r \cap {}^w\bbM_r}^{{}^w \bbM_r}$ through $\ad(w) \from \bbM_r^\sigma \to {}^w \bbM_r^\sigma$.
  %where $\bbM_r^\sigma$ acts on $X_{\bbT_r \cap {}^w \bbM_r, \bbB_r \cap {}^w \bbM_r}^{{}^w \bbM_r}$ by left-multiplication composed with $\ad(w)$.
  If $w$ does not have a representative in $\cS(\bbT_r, \bbM_r)^\sigma$, then $\sum_{i \geq 0} (-1)^i H_c^i(\widehat \Sigma_w'', \overline \QQ_\ell) = 0$.
\end{proposition}

%Using Lemma \ref{lem:Sigma hat} and Propositions \ref{prop:Sigma'}, \ref{prop:Sigma''}, we are now ready to prove the main result of this section. The proof of Theorem \ref{thm:generic mackey} appears in Section \ref{sec:generic mackey proof}

\subsection{Proof of Proposition \ref{prop:Sigma'}}\label{subsec:Sigma' proof}

The proof is a natural generalization of the arguments in \cite{Lus04,Sta09,CI21-RT}. Following \cite[Section 3.5, esp.\ (3.7)]{CI21-RT}, we have a stratification into locally closed subsets
\begin{equation*}
  \bbK_{w,0+:r+} \smallsetminus \{1\} = \bigsqcup_{1 \leq a \leq r} \bigsqcup_{I \in \cX} \bbK_{w,r}^{a,I},
\end{equation*}
where $\cX$ is the set of nonempty subsets of $\{\beta \in \Phi(\bfG,\bfT) \smallsetminus \Phi(\bfM,\bfT) : \bbU_{\beta,r} \subset \bbK_{w,r}\}$ where $\bfU_\beta$ is the root subgroup of $\bfG$ corresponding to $\beta \in \Phi(\bfG,\bfT)$. By pulling back along the natural projection $\widehat \Sigma_w' \to \bbK_{w,0+:r+} \smallsetminus \{1\}$, we have an induced stratification
\begin{equation*}
  \widehat \Sigma_w' = \bigsqcup_{a,I} \widehat \Sigma_w^{\prime, a, I}.
\end{equation*}

Fix a pair $(a,I)$ with $1 \leq a \leq r$ and $I \in \cX$.  Consider the morphism
\begin{equation*}
  \widehat \Sigma_w^{\prime, a, I} \to \bbM_{0}, \qquad (x,x',u,u',z,\mu) \mapsto \mu \bbM_{0+:r+}.
\end{equation*}
Let $\widehat \Sigma_{w,\bar\mu}^{\prime, a, I}$ denote the fiber over $\bar \mu = \mu\bbM_{0+:r+} \in \bbM_{0}$.

\begin{lemma}\label{lem:w mu action}
  Let $\alpha \in \Phi(\bfG,\bfT)$ be such that $-\alpha \in I$. Then
  $\widehat \Sigma_{w,\bar \mu}^{\prime, a, I}$ has an action of the algebraic group
  \begin{equation*}
    \cH_{\bar \mu} \colonequals \{m \in \bbM_{r:r+} : m \sigma(m)^{-1} \in \mu^{-1} \dot w^{-1} \bbT_{r:r+}^\alpha \dot w \mu\}.
  \end{equation*}
\end{lemma}

\begin{proof}
  Choose any $z \in \bbK_{w,r}^{a,I}$. For any $\xi \in \bbU_{\alpha,r-a:r+}$, consider the commutator $[\xi^{-1},z^{-1}] = \xi^{-1} z^{-1} \xi z$. By \cite[Proposition 3.8]{CI21-RT} together with the fact that $\xi$ and $z$ are both normalized by $\dot w \bbM_r \dot w^{-1}$, the construction of $\bbK_{w,r}^{a,I}$ ensures that $[\xi^{-1},z^{-1}] \in [\bbU_{\alpha,r-a:r+},\bbK_{w,r}^{a,I}]$ takes values in $\bbT_{r:r+}^\alpha(\dot w \bbN_{r:r+} \dot w^{-1})$, where $\bfT^\alpha$ is rank-1 subtorus of $\bfT$ contained in group generated by $\bfU_\alpha$ and $\bfU_{-\alpha}$. In particular, we may now define 
  \begin{equation*}
    [\xi^{-1},z^{-1}] = \tau_{\xi,z} \cdot \omega_{\xi,z}, \qquad \text{where $\tau_{\xi,z} \in \bbT_{r:r+}^\alpha$ and $\omega_{\xi,z} \in \dot w \bbN_{r:r+} \dot w^{-1}$.}
  \end{equation*} 
  Moreover, the assignment $\xi \mapsto \tau_{\xi,z}$ defines a map $\lambda_z \from \bbU_{\alpha,r-a:r+} \to \bbT_{r:r+}^\alpha$ which factors through an isomorphism $\bbU_{\alpha,r-a:(r-a)+} \cong \bbT_{r:r+}^\alpha$. Fix a section $s_z \from \bbT_{r:r+}^\alpha \to \bbU_{\alpha,r-a:r+}$ of $\lambda_z$.

  For notational convenience, write $\cH \colonequals \cH_{\bar \mu}$. For $m \in \cH$, consider the function
  \begin{equation}\label{eq:m action}
    f_{m}(x,x',u,u',z,\mu) = (x \sigma(\xi), \hat x', u, \sigma(m)^{-1}u'\sigma(m), z, \mu \sigma(m))
  \end{equation}
  for $(x,x',u,u',z,\mu) \in \widehat \Sigma_{w,\bar \mu}^{\prime,a,I}$, where
  \begin{equation*}
    \xi = s_z(\dot w \mu m \sigma(m)^{-1} \mu^{-1} \dot w^{-1}) \in \bbU_{\alpha,r-a:r+} \subset \bbU_r \cap \dot w \bbN_r \dot w^{-1}
  \end{equation*}
  and $\hat x'$ is defined by 
  $x \sigma(\xi) \sigma(z) \sigma(\dot w) \sigma(\mu) \sigma^2(m) = u z \dot w \mu u' \sigma(m) \hat x'$.   
  
  It is a quick argument to see that $f_{m'} \circ f_{m} = f_{mm'}$. Indeed, in the first coordinate, this amounts to observing that $\bbM_{r:r+}$ is commutative, and in coordinates 3 through 6, it is obvious. It follows from this that $f_{m'} \circ f_{m} = f_{mm'}$ also holds in the second coordinate. Hence to see that $f_m$ defines an action on $\widehat \Sigma_{w,\bar \mu}'$, it remains to show that the image under $f_m$ of any $(x,x',u,u',z,\mu) \in \widehat \Sigma_{w, \bar \mu}'$ lies in $\widehat \Sigma_{w, \bar \mu}'$. To do this amounts to showing $\hat x' \in \sigma(\bbN_r)$, and we spend the rest of the proof doing this.

  The argument to show $\hat x' \in \sigma(\bbN_r)$ is exactly the same as in \cite[p.\ 7]{Lus04}. We provide it here for completeness. The statement
  \begin{equation*}
    x \sigma(\xi) \sigma(z \dot w \mu \sigma(m)) \in u z \dot w \mu u' \sigma(m) \sigma(\bbN_r)
  \end{equation*}
  holds if and only if
  \begin{equation*}
    x \sigma(z) \sigma(\xi) \sigma(\tau_{\xi,z}) \sigma(\omega_{\xi,z}) \sigma(\dot w \mu \sigma(m)) \in u z \dot w \mu u' \sigma(m) \sigma(\bbN_r) 
  \end{equation*}
  since by definition $\xi z = z \xi \tau_{\xi, z} \omega_{\xi, z}$ where $\tau_{\xi, z} \in \bbT_{r:r+}^\alpha$ and $\omega_{\xi,z} \in \dot w \bbN_{r:r+} \dot w^{-1}$. By definition, we have $x\sigma(z) = u z \dot w \mu u' x' \sigma(\mu)^{-1} \sigma(\dot w)^{-1}$, so the previous statement holds if and only if
  \begin{equation*}
    x' \sigma(\mu)^{-1} \sigma(\dot w)^{-1} \sigma(\xi) \sigma(\tau_{\xi, z}) \sigma(\omega_{\xi,z}) \sigma(\dot w \mu \sigma(m)) \in \sigma(m) \sigma(\bbN_r).
  \end{equation*}
  By construction, $x' \in \sigma(\bbN_r)$ and $\sigma(\dot w^{-1}) \sigma(\omega_{\xi,z}) \sigma(\dot w) \in \sigma(\bbN_r)$ and $\sigma(\dot w^{-1}) \sigma(\xi) \sigma(\dot w) \in \sigma(\bbN_r)$. Since $\bbM_r$ normalizes $\bbN_r$, the previous statement holds if and only if
  \begin{equation*}
    \sigma(\mu)^{-1} \sigma(\dot w)^{-1} \sigma(\tau_{\xi,z}) \sigma(\dot w) \sigma(\mu \sigma(m)) \in \sigma(m) \sigma(\bbN_r),
  \end{equation*}
  and projecting to the Levi component $\bbM_r$, we see that the previous statement holds if and only if
  \begin{equation*}
    \mu^{-1} \dot w^{-1} \tau_{\xi,z} \dot w \mu = m \sigma(m)^{-1},
  \end{equation*}
  which follows from the definition of $\xi$.
\end{proof}

We can find $n \geq 1$ such that $\sigma^{n}(\mu^{-1} \dot w^{-1} \bbT_{r:r+}^\alpha \dot w \mu) = \mu^{-1} \dot w^{-1} \bbT_{r:r+}^\alpha \dot w \mu$. Then we have a morphism
\begin{equation*}
  \cN_\sigma^{\sigma^n} \from \mu^{-1} \dot w^{-1} \bbT_{r:r+}^\alpha \dot w \mu \to \cH, \qquad m \mapsto m \sigma(m) \sigma^{2}(m) \cdots \sigma^{n-1}(m)
\end{equation*}
since
\begin{equation*}
  \cN_\sigma^{\sigma^n}(m) \sigma(\cN_\sigma^{\sigma^n}(m))^{-1} =  \cN_\sigma^{\sigma^n}(m\sigma(m)^{-1}) = t' \sigma^n(m)^{-1} \in \mu^{-1} \dot w^{-1} \bbT_{r:r+}^\alpha \dot w \mu,
\end{equation*}
where the second equality holds since $\bbM_{r:r+}$ is commutative. 

\begin{lemma}\label{lem:H^0}
  The intersection $\cH^0 \cap \bbM_{r:r+}^\sigma$ contains $\cN_\sigma^{\sigma^n}((\mu^{-1} \dot w^{-1} \bbT_{r:r+}^\alpha \dot w \mu)^{\sigma^n})$.
\end{lemma}

\begin{proof}
Since $\bbT_{r:r+}^\alpha$ is connected, its image in $\cH$ under $\cN_\sigma^{\sigma^n}$ must also be connected. If $m \in (\mu^{-1} \dot w^{-1} \bbT_{r:r+}^\alpha \dot w \mu)^{\sigma^n}$, then $\cN_\sigma^{\sigma^n}(m)$ is $\sigma$-stable, so the desired conclusion follows.
\end{proof}

By Lemma \ref{lem:w mu action}, the connected algebraic group $\cH^0$ acts on $H_c^i(\widehat \Sigma_{w,\bar\mu}^{\prime, a, I}, \overline \QQ_\ell)$, and by general principles this action must be trivial. Hence by Lemma \ref{lem:H^0}, we know that the finite group $\cN_\sigma^{\sigma^n}((\mu^{-1} \dot w^{-1} \bbT_{r:r+}^\alpha\dot w \mu)^{\sigma^n})$ acts trivially on $H_c^i(\widehat \Sigma_{w,\bar\mu}^{\prime, a, I}, \overline \QQ_\ell)$. On the other hand, by construction, we have $\dot w^{-1} \cdot \alpha \notin \Phi(\bfM,\bfT)$, so the weak $(\bfM,\bfG)$-genericity of $\psi$ implies that $\psi \circ \cN_\sigma^{\sigma^n}$ is nontrivial on $\cN_\sigma^{\sigma^n}((\mu^{-1} \dot w^{-1} \bbT_{r:r+}^\alpha \dot w \mu)^{\sigma^n})$. Therefore,
\begin{equation*}
  H_c^i(\widehat \Sigma_{w, \bar \mu}^{\prime, a, I}, \overline \QQ_\ell)_{(\psi)} = 0 \qquad \text{for all $i \geq 0$}.
\end{equation*} 
Since $\bar \mu, a, I$ are all chosen arbitrarily, the conclusion of the proposition follows.

\subsection{Proof of Proposition \ref{prop:Sigma''}}\label{subsec:Sigma'' proof}
  By the Deligne--Lusztig fixed-point formula \cite[Theorem 3.2]{DL76}, if $H$ is any algebraic torus which acts on $\widehat \Sigma_w''$ compatibly with the action of $\bbM_r^\sigma \times \bbT_r^\sigma$, then we have an isomorphism
  \begin{equation*}
    \sum_{i \geq 0} (-1)^i H_c^i(\widehat \Sigma_w'', \overline \QQ_\ell) \cong     \sum_{i \geq 0} (-1)^i H_c^i((\widehat \Sigma_w'')^{H}, \overline \QQ_\ell)
  \end{equation*}
  of virtual $(\bbM_r^\sigma \times \bbT_r^\sigma)$-representations. In this proof, we will construct such an algebraic torus (we will call it $\bar H_w^0$) and show that either $(\widehat \Sigma_w'')^{\bar H_w^0}$ is empty or has cohomology equal (up to an even shift) to the cohomology of the parahoric Deligne--Lusztig variety $X_{\bbT_r, \bbB_r \cap {}^w \bbM_r}^{{}^w \bbM_r}$. 

  Recall that $(x,x',u,u',1,\mu) \in \widehat \Sigma_w''$ if and only if $x\sigma(\dot w \mu) = u \dot w \mu u' x'$. Then a straightforward calculation shows that equation \eqref{eq:action} also defines an action of
  \begin{equation*}
    H_w \colonequals \{(t,m) \in \bbT_r \times Z(\bbM_r) : 
    t^{-1}\sigma(t) = \sigma(\dot w) m^{-1}\sigma(m) \sigma(\dot w)^{-1}\}
  \end{equation*}
  on $\widehat \Sigma_w''$. Let $\bar H_w$ denote the image of $H_w$ under the surjection $\bbT_r \times Z(\bbM_r) \to \bbT_0 \times Z(\bbM_0)$. Then the identity component $\bar H_w^0$ is an algebraic torus.

  \begin{claim}\mbox{}
    \begin{enumerate}[label=(\alph*)]
      \item The projection map $\bar H_w^0 \to Z(\bbM_0)$ has image containing $Z(\bbM_0)^0$.
      \item If $w$ has a representative $\dot w$ in $\cS(\bbT_r,\bbM_r)^\sigma$, then $(\widehat \Sigma_w'')^{\bar H_w^0} \cong S_w$, where
      \begin{equation*}
        S_w \colonequals \{(u,\mu) \in (\bbU_r \cap \dot w \bbM_r \dot w^{-1}) \times \dot w \bbM_r \dot w^{-1} : u \mu \sigma(\mu)^{-1} \in \sigma(\bbU_r)\}.
      \end{equation*}
      Otherwise, $(\widehat \Sigma_w'')^{\bar H_w^0} = \varnothing.$
    \end{enumerate}
  \end{claim}

  \begin{proof}[Proof of Claim]
    For (a): Let $m \in Z(\bbM_r)^0$. Then $m^{-1}\sigma(m) \in Z(\bbM_r)^0 \subset \bbT_r$ and of course $\sigma(\dot w)^{-1} m^{-1} \sigma(m) \sigma(\dot w) \in \bbT_r$, so by the surjectivity of the Lang map, there is some $t \in \bbT_r$ such that $(t,m) \in H_w$. Hence the image of $H_w$ in the projection to $Z(\bbM_r)$ contains $Z(\bbM_r)^0$, and the same is true of $\bar H_w \to Z(\bbM_0)$. But now the connectedness of $Z(\bbM_0)^0$ implies that $\bar H_w^0$ projects surjectively onto $Z(\bbM_0)^0$.

    For (b): We compute on $\overline \FF_q$-points. Assume $(\widehat \Sigma_w'')^{\bar H_w^0} \neq \varnothing$ and let $\dot w$ be any representative of $w$. Then in particular there exists a $\mu \in \bbM_r$ such that $\dot w^{-1} t \dot w m \mu = \mu$ for all $(t,m) \in \bar H_w^0$, which implies that $\dot w^{-1} t \dot w = m$ for all $(t,m) \in \bar H_w^0$. On the other hand, by (a), this implies that
    \begin{equation*}
      \bar H_w^0 = \{(\dot w m  \dot w^{-1}, m) : m \in Z(\bbM_0)^0\}.
    \end{equation*}
    This then implies that for any $(x,x',u,u',1,\mu) \in (\widehat \Sigma_w'')^{\bar H_w^0}$, the elements $x,u$ centralize $\dot w Z(\bbM_0)^0 \dot w^{-1}$ and the elements $x',u'$ centralize $Z(\bbM_0)^0$. Since we have $Z_{\bbG_r}(Z(\bbM_0)^0) = \bbM_r$, we see that $x,u \in \dot w \bbM_r \dot w^{-1}$ and $x',u' \in \bbM_r$. Since $\bbN_r \cap \bbM_r = \{1\}$, this implies $x' = u' = 1$ and $x\sigma(\dot w \mu) = u \dot w \mu$. This implies that $w$ can be represented by an element $\dot w'$ such that $\sigma(\dot w')\dot w'{}^{-1} \in \bbM_r$, which implies that the double coset $w$ has a representative in $\cS(\bbT_r, \bbM_r)^\sigma$. From this argument, plus a simple elementary manipulation of terms, we now see (b) of the Claim.
  \end{proof}

  Part (b) of the Claim implies the last sentence of Proposition \ref{prop:Sigma''}. Now assume that $w$ has a representative $\dot w \in \cS(\bbT_r,\bbM_r)^\sigma$. We can see that the multiplication map 
  \begin{equation*}
    S_w \to {}^w \bbM_r, \qquad (u,\mu) \mapsto (u \mu)^{-1}
  \end{equation*}
  has image exactly equal to $X_{\bbT_r,\bbB_r \cap {}^w \bbM_r}^{{}^w \bbM_r}$ and fibers isomorphic to $\bbU_r \cap {}^w \bbM_r$, an affine space. Moreover, the action of $(t,m) \in \bbT_r^\sigma \times {}^w\bbM_r^\sigma$ on $(u,\mu)$ is $(tut^{-1}, t \mu \dot w m^{-1} \dot w^{-1})$, which under the above multiplication map gets sent to $(u \mu)^{-1} \mapsto \dot w m \dot w^{-1} (u \mu)^{-1} t^{-1}$, which is exactly the $(\bbT_r^\sigma \times {}^w \bbM_r^\sigma)$-action on $X_{\bbT_r, \bbB_r \cap {}^w \bbM_r}^{{}^w \bbM_r}$.

  \subsection{Proof of Theorem \ref{thm:generic mackey}}\label{subsec:generic mackey proof}

  Let $\rho$ be a weakly $(\bfM,\bfG)$-generic representation of $\bbM_r^\sigma$. The desired result follows directly from Propositions \ref{prop:Sigma'} and \ref{prop:Sigma''}; we spell it out in detail here. For any $t \in \bbT_r^\sigma$, we have
  \begin{align*}
    {}^* {}&{}R_{\bbT_r,\bbB_r}^{\bbG_r} \circ R_{\bbM_r,\bbP_r}^{\bbG_r}(\rho)(t) \\
    &= \frac{1}{|\bbM_r^\sigma|} \sum_{m \in \bbM_r^\sigma} \rho(m)^{-1} \Tr((t,m); H_c^*(\Sigma, \overline \QQ_\ell)) \\
    &= \sum_{w \in \bbT_r \backslash \cS(\bbT_r, \bbM_r)^\sigma / \bbM_r} \frac{1}{|\bbM_r^\sigma|} \sum_{m \in \bbM_r^\sigma} \rho(m)^{-1} \Tr((t,m); H_c^*(\widehat \Sigma_w, \overline \QQ_\ell)) \\
    &= \sum_{w \in \bbT_r \backslash \cS(\bbT_r, \bbM_r)^\sigma / \bbM_r} \frac{1}{|\bbM_r^\sigma|} \sum_{m \in \bbM_r^\sigma} \rho(m)^{-1} \Tr((t,m); H_c^*(\widehat \Sigma_w'', \overline \QQ_\ell)) \\
    &= \sum_{w \in \bbT_r \backslash \cS(\bbT_r, \bbM_r)^\sigma / \bbM_r} \frac{1}{|\bbM_r^\sigma|} \sum_{m \in \bbM_r^\sigma} \rho(m)^{-1} \Tr((\dot w m \dot w^{-1},t); H_c^*(X_{\bbT_r, \bbB_r \cap {}^w \bbM_r}^{{}^w \bbM_r}, \overline \QQ_\ell)) \\
    &= \sum_{w \in \bbT_r \backslash \cS(\bbT_r, \bbM_r)^\sigma / \bbM_r} {}^* R_{\bbT_r, \bbB_r \cap {}^w \bbM_r}^{{}^w \bbM_r}(\ad(w)^{-1}(\rho))(t).
  \end{align*}
  where the first equality follows from the isomorphism $X_{\bbT_r,\bbB_r}^{\bbG_r} \times_{\bbG_r^\sigma} X_{\bbM_r,\bbP_r}^{\bbG_r} \cong \Sigma$, the second equality follows from Lemma \ref{lem:Sigma hat}, the third equality follows from Proposition \ref{prop:Sigma'}, the fourth equality follows from Proposition \ref{prop:Sigma''}, and the last equality holds by definition.

  % From Propositions \ref{prop:Sigma'} and \ref{prop:Sigma''}, we know that the variety $\Sigma$ defining the composition ${}^* R_{\bbT_r,\bbB_r}^{\bbG_r} \circ R_{\bbM_r, \bbP_r}^{\bbG_r}$ has the same cohomology as a disjoint union varieties $\widehat \Sigma_w''$ each of which defines the composition $\ad(w) \circ ^*R_{{}^w \bbT_r, {}^w \bbB_r \cap \bbM_r}^{\bbM_r}$. \textcolor{red}{add genericity at this point} We therefore obtain
  % \begin{equation}
  %   {}^* R_{\bbT_r,\bbB_r}^{\bbG_r} \circ R_{\bbM_r,\bbP_r}^{\bbG_r} = \sum_{w \in \bbT_r \backslash \cS(\bbT_r, \bbM_r)^\sigma/\bbM_r} \ad(w) \circ {}^* R_{{}^w \bbT_r, {}^w \bbB_r \cap \bbM_r}^{\bbM_r}.
  % \end{equation}
  
\section{Parahoric Deligne--Lusztig varieties for elliptic tori}\label{sec:depth compatibility}

It is natural to ask how the parahoric Deligne--Lusztig functors $R_{\bbT_r,\bbB_r}^{\bbG_r}$ are compatible as $r$ varies. From the surjectivity of the Lang map, it follows that the morphism
\begin{equation*}
  \tilde \pi \from X_{\bbT_r,\bbB_r}^{\bbG_r} \to X_{\bbT_{r-1},\bbB_{r-1}}^{\bbG_{r-1}}
\end{equation*}
is surjective. The technical effort of this section is in proving the following theorem:

\begin{theorem}\label{thm:fiber cohomology}
  Let $N = \dim \bbU_{(r-1)+:r+}$. 
  For any point $x \in X_{\bbT_{r-1},\bbB_{r-1}}^{\bbG_{r-1}}$,
  \begin{equation*}
    H_c^i(\tilde \pi^{-1}(x), \overline \QQ_\ell)^{\bbT_{r:r+}^\sigma} = \begin{cases}
      \overline \QQ_\ell^{\oplus \#\bbU_{(r-1)+:r+}^\sigma} & \text{if $i = 2N$,} \\
      0 & \text{otherwise.}
    \end{cases}
  \end{equation*} 
  Moreover, $\sigma^n$ acts on $H_c^{2N}(\tilde \pi^{-1}(x), \overline \QQ_\ell)^{\bbT_{r:r+}^\sigma}$ by multiplication by $q^{nN}$.
\end{theorem}

We prove this theorem in Section \ref{subsec:fiber cohomology}. The techniques we employ also work to calculate the cohomology of the fibers of the depth-lower projections of parahoric Lusztig varieties $X_{\bbM_r,\bbP_r}^{\bbG_r}$. The answer is the same as in Theorem \ref{thm:fiber cohomology}, with $(M,P,N)$'s replacing $(T,B,U)$'s (preserving font).

Theorem \ref{thm:fiber cohomology} has some important immediate corollaries. If $\bfT$ is elliptic over $F$, then $\bbU_{r:r+}^\sigma = \{1\}$, and so we obtain the following result as a corollary of Theorem \ref{thm:fiber cohomology}.

\begin{theorem}\label{thm:level lower}
  If $\bfT$ is elliptic over $F$, then we have $(\bbG_r^\sigma \times \bbT_r^\sigma)$-equivariant isomorphisms 
  \begin{equation}\label{eq:level lower}
    H_c^i(X_{\bbT_r,\bbB_r}^{\bbG_r}, \overline \QQ_\ell)^{\bbT_{r:r+}^\sigma} \cong H_c^{i+2N}(X_{\bbT_{r-1},\bbB_{r-1}}^{\bbG_{r-1}}, \overline \QQ_\ell(N)) \qquad \text{for all $i \geq 0$}.
  \end{equation}
\end{theorem}

In the above, $(N)$ denotes the Tate twist. By Theorem \ref{thm:level lower}, we immediately obtain the following corollaries.

\begin{corollary}\label{cor:depth compatibility}
  Assume $\bfT \subset \bfG$ is elliptic over $F$ and fix $s < r$. For any character $\theta \from \bbT_s^\sigma \to \overline \QQ_\ell^\times$, we have an isomorphism of virtual $\bbG_r^\sigma$-representations
  \begin{equation*}
    R_{\bbT_r,\bbB_r}^{\bbG_r}(\theta) \cong R_{\bbT_s,\bbB_s}^{\bbG_s}(\theta).
  \end{equation*}
\end{corollary}

Following \cite{Lus79}, define $H_i(S, \overline \QQ_\ell) \colonequals H_c^{2\dim(S)-i}(S, \overline \QQ_\ell(\dim(S)))$ for any smooth $\overline \FF_q$-scheme of pure dimension. 

\begin{corollary}\label{cor:infinite level}
  Assume $\bfT \subset \bfG$ is elliptic over $F$. We have a natural embedding
  \begin{equation*}
    H_i(X_{\bbT_{r-1},\bbB_{r-1}}^{\bbG_{r-1}}, \overline \QQ_\ell) \hookrightarrow H_c^i(X_{\bbT_r,\bbB_r}^{\bbG_r}, \overline \QQ_\ell).
  \end{equation*}
  For $X_{\bbT_\infty,\bbB_\infty}^{\bbG_\infty} \colonequals \varprojlim_r X_{\bbT_r,\bbB_r}^{\bbG_r}$, setting
  \begin{equation*}
    H_i(X_{\bbT_\infty,\bbB_\infty}^{\bbG_\infty}, \overline \QQ_\ell) \colonequals \varinjlim_r H_i(X_{\bbT_r,\bbB_r}^{\bbG_r}, \overline \QQ_\ell)
  \end{equation*}
  therefore defines $\ell$-adic homology groups for the infinite-dimensional $\overline \FF_q$-scheme $X_{\bbT_\infty,\bbB_\infty}^{\bbG_\infty}$. Moreover, on the category of smooth representations of $\bbT_\infty^\sigma$, it makes sense to define a functor $R_{\bbT_\infty,\bbB_\infty}^{\bbG_\infty}$ analogously to Definition \ref{def:induction}, and for any character $\theta$ of $\bbT_r^\sigma$, we have an equality of $\bbG_\infty^\sigma$-representations
  \begin{equation*}
    R_{\bbT_\infty,\bbB_\infty}^{\bbG_\infty}(\theta) = R_{\bbT_r,\bbB_r}^{\bbG_r}(\theta)
  \end{equation*}
\end{corollary}

\begin{remark}\label{rem:relation to lusztig}
  In 1979, Lusztig conjectured \cite{Lus79} that there should exist reasonable $p$-adic Deligne--Lusztig spaces. Lusztig studied this in \textit{op.\ cit. }for $\bfG$ the norm-1 elements of division algebras, and this was later formalized and generalized by Boyarchenko \cite{Boy12} to $\bfG$ a division algebra. For other inner forms of $\GL_n$, this was studied by the author and Ivanov \cite{CI21-MA,CI_loopGLn}. In these settings, representation-theoretic calculations proceed by establishing:
  \begin{enumerate} 
    \item The $p$-adic Deligne--Lusztig space is a disjoint union of infinite-dimensional parahoric Deligne--Lusztig varieties $X_\infty$.
    \item $\ell$-adic homology groups of $X_\infty$ can be defined as a direct limit of $\ell$-adic homology groups of finite-depth parahoric Deligne--Lusztig varieties $X_r$.
  \end{enumerate}
  For $\GL_n$, elliptic unramified maximal tori are automatically Coxeter, but this is no longer the case for general connected reductive groups $\bfG$; on the other hand, progress on (1) has only been made in the Coxeter setting. For $\bfG = \GSp$ and $\bfT$ Coxeter, Takamatsu established (1) in \cite{T23}. For $\bfG$ unramified of classical type and $\bfT$ Coxeter, Ivanov proved (1) in \cite{I23arc,I23orbit}. In all these settings, the parahoric schemes $X_\infty$ are examples of $X_{\bbT_\infty,\bbB_\infty}^{\bbG_\infty}$, hence Corollary \ref{cor:infinite level} resolves (2) and endows the above studied infinite-dimensional $p$-adic Deligne--Lusztig spaces with $\ell$-adic homology groups. This therefore generalizes the definition of homology in \cite{Lus79,Boy12,CI21-MA,CI_loopGLn} to Ivanov's setting in \cite{I23orbit} and relates results on the cohomology of finite-depth parahoric Deligne--Lusztig varieties---for example of the author and Oi \cite{CO21}---to the setting of Lusztig's 1979 conjecture.
\end{remark}

\begin{remark}\label{rem:indirect depth compat}
  We offer an indirect alternate argument to the discussion in Remark \ref{rem:relation to lusztig}. Another way to endow the $p$-adic Deligne--Lusztig spaces in Ivanov's decomposition result \cite{I23orbit} (for $\bfG$ of classical type and Coxeter $\bfT$) with $\ell$-adic homology groups is to use Dudas--Ivanov \cite{DI20} (scalar product formula for Coxeter $\bfT$ and $q>5$) in tandem with the results of the author with Oi \cite{CO21} (arbitrary $\bfT$, $q \gg 0$, genericity condition on $\theta$) which characterizes the irreducible representations $R_{\bbT_r}^{\bbG_r}(\theta)$. (As mentioned in the introduction, in forthcoming work, the author and Oi will remove the genericity condition in \cite{CO21}.) \textit{A posteriori}, we then obtain $R_{\bbT_r}^{\bbG_r}(\theta) \cong R_{\bbT_{r+1}}^{\bbG_{r+1}}(\theta)$ when $q \gg 0$.
\end{remark}

\subsection{The cohomology of the fibers of $\tilde \pi$}\label{subsec:fiber cohomology}

The purpose of this section is to prove Theorem \ref{thm:fiber cohomology}. The simplest reason for this theorem to hold would be if $\tilde \pi^{-1}(x)/\bbT_{r:r+}^\sigma \cong \bbA^d$ (as usual, up to perfection). This is the case when $\bfG$ is a division algebra \cite[Lemma 4.7]{Boy12} and when $\bfG$ is more generally any inner form of $\GL_n$ \cite[Proposition 7.6]{CI21-MA}. While this is true at least for some $x \in X_{\bbT_{r-1},\bbB_{r-1}}^{\bbG_{r-1}}$ (for example if the image of $x$ in $\bbG_0 = \bbG_{0:0+}$ is $\FF_q$-rational), despite our best efforts over several years, we were not able to prove this isomorphism for arbitrary $x$. In the following, we focus instead on the statement of Theorem \ref{thm:fiber cohomology}, which requires only a calculation about the cohomology of $\tilde \pi^{-1}(x)$, not its explicit geometry.

For notational convenience, let us prove the theorem for $X_{0:r+} \colonequals X_{\bbT_r,\sigma^{-1}(\bbB_r)}^{\bbG_r}$. Set $X_{0:r} \colonequals \{g \in \bbG_{0:r} : g^{-1}\sigma(g) \in \bbU_{0:r}\}$. Choose any $\tilde x \in X_{0:r+}$ over $x \in X_{0:r}$. Denote by $\bar x$ the image of $x$ in $\bbG_{0:0+}$. We have a morphism
\begin{equation*}
  X_{0:r+}^{\bbG_r}/\bbT_{r:r+}^\sigma \to \{g \in \bbG_{0:r+} : g^{-1} \sigma(g) \in \bbU_{0:r+}\bbB_{r:r+}\}/\bbB_{r:r+} \cap \sigma^{-1}(\bbB_{r:r+})
\end{equation*}
whose fibers are isomorphic to the affine space $\bbU_{r:r+} \cap \sigma^{-1}(\bbU_{r:r+})$. Since $\bbB_{r:r+} \cap \sigma^{-1}(\bbB_{r:r+})$ is also an affine space, we see that the cohomology of $\tilde \pi^{-1}(x)/\bbT_{r:r+}^\sigma$ is, up to a shift of $2 \dim \bbT_{r:r+}$, equal to the cohomology of the fibers of 
\begin{equation*}
  \pi \from \{g \in \bbG_{0:r+} : g^{-1} \sigma(g) \in \bbU_{0:r+}\bbB_{r:r+}\} \to \{g \in \bbG_{0:r} : g^{-1} \sigma(g) \in \bbU_{0:r}\}.
\end{equation*}

For any $x_r \in \bbG_{r:r+}$, we have
\begin{align*}
  (\tilde x x_r)^{-1} \sigma(\tilde x x_r) \in \bbU_{0:r+}\bbB_{r:r+} 
  &\Longleftrightarrow x_r^{-1} \tilde x^{-1} \sigma(\tilde x) \sigma(x_r) \sigma(\tilde x)^{-1} \tilde x \in \bbB_{r:r+} \\
  &\Longleftrightarrow (\tilde x x_r \tilde x^{-1})^{-1} \sigma(\tilde x x_r \tilde x^{-1}) \in \tilde x \bbB_{r:r+} \tilde x^{-1}.
\end{align*}
Note that $\tilde x \bbU_{r:r+} \tilde x^{-1}$ only depends on $x$ (in fact, only on $\bar x$); hence we write $x \bbU_{r:r+} x^{-1}$ for this subgroup. We have shown that we have an isomorphism
\begin{equation*}
  \pi^{-1}(x) \cong \{x_r \in\bbG_{r:r+} : x_r^{-1} \sigma(x_r) \in x \bbB_{r:r+} x^{-1}\}.
\end{equation*}
The Lang map
\begin{equation*}
  \bbG_{r:r+} \to\bbG_{r:r+}, \qquad g_r \mapsto g_r^{-1} \sigma(g_r)
\end{equation*}
restricts to a morphism
\begin{equation*}
  \varphi \from \pi^{-1}(x) \to x \bbB_{r:r+} x^{-1},
\end{equation*}
and so we see that
\begin{equation*}
  \varphi_! \overline \QQ_\ell = \bigoplus_{\chi \in (\bbG_{r:r+}^\sigma)^\wedge} i^* \cL_\chi
\end{equation*}
where we write $i \from x \bbB_{r:r+} x^{-1} \hookrightarrow\bbG_{r:r+}$. By construction, each sheaf $i^* \cL_\chi$ is a multiplicative local system on the connected algebraic group $x \bbB_{r:r+} x^{-1}$ which is an affine space, and we therefore see that
\begin{equation*}
  H_c^i(x \bbB_{r:r+} x^{-1}, i^* \cL_\chi) = \begin{cases}
    \overline \QQ_\ell & \text{if $i = 2\dim \bbB_{r:r+}$ and $i^* \cL_\chi \cong \overline \QQ_\ell$,} \\
    0 & \text{otherwise.}
  \end{cases}
\end{equation*}
From this, we may conclude
\begin{equation}\label{eq:fiber cohomology}
  H_c^i(\pi^{-1}(x), \overline \QQ_\ell) = \bigoplus_\chi H_c^i(x \bbB_{r:r+} x^{-1}, i^* \cL_\chi) = \begin{cases}
    \overline \QQ_\ell^{\oplus\#A} & \text{if $i = 2\dim \bbB_{r:r+}$,} \\
    0 & \text{otherwise,}
  \end{cases}
\end{equation}
where
\begin{equation*}
  A = \{\chi \in (\bbG_{r:r+}^\sigma)^\wedge : i^* \cL_\chi \cong \overline \QQ_\ell\}.
\end{equation*}

Let $m$ be such that $\sigma^m(x) = x$ and $\sigma^m(\bbU_{r:r+}) = \bbU_{r:r+}$ (can choose $m$ to be minimal). Since $\cL_\chi$ is $\sigma$-equivariant by construction, it follows that if $i^* \cL_\chi \cong \overline \QQ_\ell$, then $(\sigma^j)^* i^* \cL_\chi \cong \overline \QQ_\ell$ for all $j$. Hence $\cL_\chi$ pulls back to the constant local system on
\begin{equation}\label{eq:xBx}
  x \bbB_{r:r+} x^{-1} \cdot \sigma(x \bbB_{r:r+} x^{-1}) \cdots \sigma^{m-1}(x \bbB_{r:r+} x^{-1}) \subseteq \bbG_{r:r+}.
\end{equation}

\begin{proposition}\label{prop:|A| elliptic}
  If $\bfT \subset \bfG$ is elliptic over $F$, then $|A| = 1$.
\end{proposition}

\begin{proof}
  To show the proposition, we will prove that equality holds in \eqref{eq:xBx}. As already noted, $x \bbB_{r:r+} x^{-1}$ depends only on $\bar x \in \bbG_{0:0+}$. By Bruhat decomposition, we have $\bar x \in \bar u w' \bbB_{0:0+}$ for some $\bar u \in \bbU_{0:r+}$ and $w' \in N_{\bbT_{0:0+}}(\bbG_{0:0+})$. Hence we have $x \bbB_{r:r+} x^{-1} = \bar u w' \bbB_{r:r+} w'{}^{-1} \bar u^{-1}$, and we may in fact assume $\bar x = \bar u w$. The adjoint action of $\bbT_{0:0+}$ on $\bbG_{r:r+}$ gives us a direct sum decomposition $\bbG_{r:r+} = \bbT_{r:r+} \oplus \bbU_{r:r+}^A \oplus \bbU_{r:r+}^{\Phi \smallsetminus A}$ for any $A \subset \Phi$, where we set
  \begin{equation*}
    \bbU_{r:r+}^{A} \colonequals \prod_{\alpha \in A} \bbU_{r:r+}^\alpha.
  \end{equation*} 
  Write $\varphi_A \from \bbG_{r:r+} \to \bbU_{r:r+}^A$ for the associated orthogonal projection. We observe that the composition
  \begin{equation*}
    \bar u w' \bbU_{r:r+}^A w^{-1} \bar u^{-1} \hookrightarrow \bbG_{r:r+} \stackrel{\varphi_{\Ad(w')A}}{\to} w' \bbU_{r:r+}^A w^{-1}
  \end{equation*}
  is surjective. Indeed, for any $x_\alpha \in \bbU_{r:r+}^\alpha$, we have $\bar u w' x_\alpha w'{}^{-1} \bar u^{-1} = [\bar u w', x_\alpha] x_\alpha$, and the $\alpha$-projection of $[\bar u w', x_\alpha]$ is zero; this shows surjectivity on $w' \bbU_{r:r+}^\alpha w'{}^{-1}$. Since $\sigma(\bbT_{0:0+}) = \bbT_{0:0+}$, this implies that for any $j$, the composition
  \begin{equation*}
    \sigma^j(x\bbU_{r:r+}^A x^{-1}) \hookrightarrow \bbG_{r:r+} \stackrel{\varphi_{\sigma^j(\Ad(w')A)}}{\to} \sigma^j(w' \bbU_{r:r+}^A w'{}^{-1})
  \end{equation*}
  is surjective. Writing $\sigma(g) = w \sigma_0(g) w^{-1}(g)$, the above results then translate to having $\sigma^j(w' \bbU_{r:r+}^A w') = \Ad(w^j w') \bbU_{r:r+}^A$ and 
  \begin{equation*}
    \varphi_{\Ad(w^jw')(A)}(\sigma^j(x \bbU_{r:r+}^A x^{-1})) = \Ad(w^j w') \bbU_{r:r+}^A.
  \end{equation*}

  Consider now
  \begin{equation*}
    A_j \colonequals \Phi^+ \cap (\cap_{1 \leq i \leq j} \Ad(w'{}^{-1} w^i w')^{-1} \Phi^-).
  \end{equation*}
  We claim that for any $j,n \geq 0$ with $j + n \leq m-1$, the image of the span of
      \begin{equation}\label{eq:n}
        x\bbU_{r:r+}^{A_j}x^{-1}, \quad \sigma(x \bbU_{r:r+}^{A_{j+1}} x^{-1}), \quad \ldots, \quad \sigma^{n}(x \bbU_{r:r+}^{A_{j+n}} x^{-1})
      \end{equation}
  under $\varphi_{A}$ is equal to $\bbU_{r:r+}^A$, where $A = \cup_{0 \leq k \leq n} \Ad(w^k w')(A_{j+k})$. We induct on $n$.

  The base case is $n=0$, and this is already clear from the first paragraph of this proof. Assume the result holds for $n-1$ so that the image of the span of
  \begin{equation*}
    x \bbU_{r:r+}^{A_{j+1}} x^{-1}, \sigma(x \bbU_{r:r+}^{A_{j+2}} x^{-1}), \ldots, \sigma^{n-1}(x \bbU_{r:r+}^{A_{j+n}} x^{-1})
  \end{equation*}
  under $\varphi_{A'}$ is equal to $\bbU_{r:r+}^{A'}$, where $A' = \cup_{0 \leq k \leq n-1} \Ad(w^k w')(A_{j+1+k})$. This implies that the image of the span of
  \begin{equation*}
    \sigma(x \bbU_{r:r+}^{A_{j+1}} x^{-1}), \sigma^2(x \bbU_{r:r+}^{A_{j+2}} x^{-1}), \ldots, \sigma^{n}(x \bbU_{r:r+}^{A_{j+n}} x^{-1})
  \end{equation*}
  under $\varphi_{\Ad(w)A'}$ is equal to $\bbU_{r:r+}^{\Ad(w)A'}$. This implies that the image of the span of
    \begin{equation}\label{e:Ad(x^{-1})}
    x^{-1}\sigma(x \bbU_{r:r+}^{A_{j+1}} x^{-1})x, x^{-1}\sigma^2(x \bbU_{r:r+}^{A_{j+2}} x^{-1})x, \ldots, x^{-1}\sigma^{n}(x \bbU_{r:r+}^{A_{j+n}} x^{-1})x
    \end{equation}
  under $\Ad(x^{-1}) \circ \varphi_{\Ad(w)(A')}$ is equal to $x^{-1}\bbU_{r:r+}^{\Ad(w)(A')}x = \Ad(w'{}^{-1} \bar u^{-1} w') \bbU_{r:r+}^{\Ad(w'{}^{-1}w)(A')}$. By construction, $\Ad(w'{}^{-1} w)(A') = \cup_{1 \leq k \leq n} \Ad(w'{}^{-1} w^k w')(A_{j+k})$ has disjoint intersection with $A_j$. Hence $\bbU_{r:r+}^{A_j}$ has trivial intersection with the span of \eqref{e:Ad(x^{-1})}, which proves linear independence. Therefore the image of the span of \eqref{eq:n} under $\varphi_{A}$ is equal to $\bbU_{r:r+}^A$.

  The ellipticity of $w$ implies that $\Phi^- = \cup_{1 \leq k \leq m-1} \Ad(w'{}^{-1} w^k w')(A_k)$. By the above, we have that the image of the span of 
  \begin{equation*}
    x^{-1} \sigma(x \bbU_{r:r+}^{A_1} x^{-1}) x, x^{-1} \sigma^2(x \bbU_{r:r+}^{A_2} x^{-1}) x, \ldots, x^{-1} \sigma^{m-1}(x \bbU_{r:r+}^{A_{m-1}} x^{-1})x
  \end{equation*}
  under $\varphi_{\Phi^-}$ is equal to $\bbU_{r:r+}^{\Phi^-}$. But now this means that the span of these subspaces with $\bbB_{r:r+}$ is equal to $\bbG_{r:r+}$.
\end{proof}

All statements above also work with the depth $r$ with any $s \in \bbR$ with $r-1 \leq s \leq r$. It therefore follows from Proposition \ref{prop:|A| elliptic} that if $\bfT \subset \bfG$ is elliptic, the conclusion of Theorem \ref{thm:fiber cohomology} holds. This then completes the proof of Theorem \ref{thm:level lower}. We will now use this to prove Theorem \ref{thm:fiber cohomology} in the general case. 

Let $\bfL,\bfQ$ be as in Lemma \ref{lem:pInd factor} so that $R_{\bbT_r,\bbB_r}^{\bbG_r} = \Ind_{\bbQ_{0:r+}^\sigma}^{\bbG_{0:r+}^\sigma} \circ \Inf_{\bbL_{0:r+}^\sigma}^{\bbQ_{0:r+}^\sigma} \circ R_{\bbT_{0:r+},\bbL_{0:r+} \cap \bbB_{0:r+}}^{\bbL_{0:r+}}$. Then for any $\theta$ of depth $<r$, we have 
\begin{align*}
  &\Inf_{\bbG_{0:(r-1)+}^\sigma}^{\bbG_{0:r+}^\sigma} R_{\bbT_{0:(r-1)+},\bbB_{0:(r-1)+}}^{\bbG_{0:(r-1)+}}(\theta) \\
  &= \Ind_{\bbQ_{0:(r-1)+}^\sigma}^{\bbG_{0:(r-1)+}^\sigma}(\Inf_{\bbL_{0:(r-1)+}^\sigma}^{\bbQ_{0:(r-1)+}^\sigma}(R_{\bbT_{0:(r-1)+},\bbL_{0:(r-1)+} \cap \bbB_{0:(r-1)+}}^{\bbL_{0:(r-1)+}}(\theta))) \\
  &= \Ind_{\bbQ_{0:r+}^\sigma \bbG_{({r-1})+:r+}^\sigma}^{\bbG_{0:r+}^\sigma}(\Inf_{\bbL_{0:(r-1)+}^\sigma}^{\bbQ_{0:r+}^\sigma \bbG_{({r-1})+:r+}^\sigma}(R_{\bbT_{0:(r-1)+},\bbL_{0:(r-1)+} \cap \bbB_{0:(r-1)+}}^{\bbL_{0:(r-1)+}}(\theta))) \\
  &= \Ind_{\bbQ_{0:r+}^\sigma \bbG_{({r-1})+:r+}^\sigma}^{\bbG_{0:r+}^\sigma}(\Inf_{\bbL_{0:r+}^\sigma}^{\bbQ_{0:r+}^\sigma \bbG_{({r-1})+:r+}^\sigma}(R_{\bbT_{0:r+},\bbL_{0:r+} \cap \bbB_{0:r+}}^{\bbL_{0:r+}}(\theta))),
\end{align*}
% \begin{align*}
%   \Inf_{\bbG_{r-1}^\sigma}^{\bbG_r^\sigma} R_{\bbT_{r-1},\bbB_{r-1}}^{\bbG_{r-1}}(\theta) 
%   &= \Ind_{\bbQ_{r-1}^\sigma}^{\bbG_{r-1}^\sigma}(\Inf_{\bbL_{r-1}^\sigma}^{\bbQ_{r-1}^\sigma}(R_{\bbT_{r-1},\bbL_{r-1} \cap \bbB_{r-1}}^{\bbL_{r-1}}(\theta))) \\
%   &= \Ind_{\bbQ_r^\sigma \bbG_{({r-1})+:r+}^\sigma}^{\bbG_r^\sigma}(\Inf_{\bbL_{r-1}^\sigma}^{\bbQ_r^\sigma \bbG_{({r-1})+:r+}^\sigma}(R_{\bbT_{r-1},\bbL_{r-1} \cap \bbB_{r-1}}^{\bbL_{r-1}}(\theta))) \\
%   &= \Ind_{\bbQ_r^\sigma \bbG_{({r-1})+:r+}^\sigma}^{\bbG_r^\sigma}(\Inf_{\bbL_r^\sigma}^{\bbQ_r^\sigma \bbG_{({r-1})+:r+}^\sigma}(R_{\bbT_r,\bbL_r \cap \bbB_r}^{\bbL_r}(\theta))),
% \end{align*}
where the last equality holds by Theorem \ref{thm:level lower}. This shows that
\begin{equation*}
  \dim R_{\bbT_r,\bbB_r}^{\bbG_r}(\theta) = |\bbU_{(r-1)+:r+}^\sigma| \cdot \dim R_{\bbT_{r-1},\bbT_{r-1}}^{\bbG_{r-1}}(\theta).
\end{equation*} 
On the other hand, \eqref{eq:fiber cohomology} implies that
\begin{equation}
  \dim H_c^i(X_{\bbT_r,\bbB_r}^{\bbG_r}, \overline \QQ_\ell)^{\bbT_{r:r+}^\sigma} = |A| \cdot \dim H_c^{i-2N}(X_{\bbT_{r-1},\bbB_{r-1}}^{\bbG_{r-1}}, \overline \QQ_\ell), \qquad \text{for all $i \geq 0$}.
\end{equation}
Therefore $|A| = |\bbU_{(r-1)+:r+}^\sigma|$, which now completes the proof of Theorem \ref{thm:fiber cohomology}.

\section{The scalar product formula for parahoric Deligne--Lusztig induction}\label{sec:scalar}

% For the rest of the paper, assume that either:
% \begin{enumerate}[label=(\alph*)] 
%   \item $(\theta,\bbT_r)$ has a Howe factorization, or 
%   \item $p$ is not a torsion prime for the root system of $\bfG$.
% \end{enumerate}
% Recall that (b) implies that any $(\theta,\bbT_r)$ has a Howe factorization (Theorem \ref{thm:Howe}; \cite[Proposition 3.6.7]{Kal19}).

\begin{definition}
  Let $\theta \from \bbT_r^\sigma \to \overline \QQ_\ell^\times$ be any character. Then the restriction $\theta|_{\bbT_{r:r+}^\sigma}$ agrees with the restriction of a weakly $(\bfM,\bfG)$-generic character for some Levi subgroup $\bfM$ of $\bfG$. 
  We say that  $(\theta,\bbT_r)$ is \textit{split-generic} if $\bfT$ is elliptic over $F$ as a torus of $\bfM$. %if $\bfM \cap \bfU \cap \sigma(\bfU) \cap \cdots \cap \sigma^{n-1}(\bfU) = \{e\}.$
\end{definition}

We now come to the main theorem of the paper.

\begin{theorem}\label{thm:scalar product formula}
  Let $(\theta,\bbT_r,\bbB_r)$ be split-generic and Howe-factorizable. For any $(\theta',\bbT_r',\bbB_r')$,
  \begin{equation*}
    \langle R_{\bbT_r,\bbB_r}^{\bbG_r}(\theta), R_{\bbT_r',\bbB_r'}^{\bbG_r}(\theta') \rangle_{\bbG_r^\sigma} = \sum_{w \in W_{\bbG_r}(\bbT_r,\bbT_r')^\sigma} \langle \theta, \ad(w^{-1})^* \theta' \rangle_{\bbT_r^\sigma}.
  \end{equation*}
  %Moreover, the assumption on $(\theta,\bbT_r,\bbB_r)$ is sharp: if $(\theta,\bbT_r,\bbB_r)$ is not split-generic, then there exists a $(\theta',\bbT_r',\bbB_r')$ such that the above formula does not hold.
  In particular, $R_{\bbT_r,\bbB_r}^{\bbG_r}(\theta)$ is independent of the choice of $\bfB$.
\end{theorem}

Let us state an immediate corollary in the setting that $\bfT \subset \bfG$ is elliptic over $F$, the notation here being as in Corollary \ref{cor:depth compatibility}.

\begin{corollary}
  Let $\bfT \subset \bfG$ be elliptic over $F$ and assume $p$ is not a torsion prime for the root system of $\bfG$.
  \begin{enumerate}
    \item The functor $R_{\bbT_\infty}^{\bbG_\infty} \colonequals R_{\bbT_\infty,\bbB_\infty}^{\bbG_\infty}$ is independent of the choice of $\bfB$.
    \item $R_{\bbT_\infty}^{\bbG_\infty}(\theta)$ is irreducible if and only if $\Stab_{W_{\bbG_\infty}(\bbT_\infty)^\sigma}(\theta) = \{1\}.$
  \end{enumerate}
\end{corollary}

\subsection{Proof of the scalar product formula}

We first note the following proposition, which comes as an easy corollary of several results we have established in this paper.

\begin{proposition}\label{prop:Howe}
  Let $(\theta,\bbT_r,\bbB_r)$ be split-generic and let $\vec \phi$ be any Howe factorization and choose any accompanying sequence $\vec \bfP$ of parabolic subgroups. Then
  \begin{equation*}
    R_{\bbT_r,\bbB_r}^{\bbG_r}(\theta) = r_{\bbT_r}^{\bbG_r}(\vec \phi; \vec \bfP).
  \end{equation*}
\end{proposition}

\begin{proof}
  Since $(\theta,\bbT_r,\bbB_r)$ is split-generic by assumption, we may apply Theorem \ref{thm:level lower} at each intermediate step. Hence we have
  \begin{align*}
    R_{\bbG_{r_1}^{1}, \bbP_{r_1}^{1}}^{\bbG_{r_1}^{2}}(R_{\bbT_{r_0}', \bbB_{r_0}'}^{\bbG_{r_0}^{1}}(\theta_0') \otimes \theta_1') 
    &= R_{\bbG_{r_1}^{1}, \bbP_{r_1}^{1}}^{\bbG_{r_1}^{2}}(R_{\bbT_{r_1}', \bbB_{r_1}'}^{\bbG_{r_1}^{1}}(\theta_0') \otimes \theta_1')\\
    &= R_{\bbG_{r_1}^{1}, \bbP_{r_1}^{1}}^{\bbG_{r_1}^{2}}(R_{\bbT_{r_1}', \bbB_{r_1}'}^{\bbG_{r_1}^{1}}(\theta_0' \otimes \theta_1')) \\
    &= R_{\bbT_{r_1}', \bbB_{r_1}'}^{\bbG_{r_1}^{2}}(\theta_0' \otimes \theta_1')
  \end{align*}
  where the first equality holds by Theorem \ref{thm:level lower}, the second equality holds by Proposition \ref{prop:twisting}, and the third equality holds by Proposition \ref{prop:transitivity}. Continuing this, we see the desired equality.
\end{proof}

With Proposition \ref{prop:Howe} in mind, Theorem \ref{thm:scalar product formula} follows from calculating the inner product $\langle R_{\bbT_r,\bbB_r}^{\bbG_r}(\theta), r_{\bbT_r'}^{\bbG_r}(\vec \phi'; \vec \bfP') \rangle,$ which we do in Proposition \ref{prop:half scalar product} below. The final assertion of Theorem \ref{thm:scalar product formula} about independence of the choice of $\bfB$ follows from the scalar product formula using the same trick as in \cite[Corollary 2.4]{Lus04}: the inner product of $R_{\bbT_r,\bbB_r}^{\bbG_r}(\theta) - R_{\bbT_r,\bbB_r'}^{\bbG_r}(\theta)$ with itself is equal to zero.

\begin{proposition}\label{prop:half scalar product}
  Let $\vec \phi'$ be any Howe factorization of $\theta'$ and choose any accompanying sequence $\vec \bfP$ of parabolic subgroups. Then 
  \begin{equation*}
    \langle R_{\bbT_r,\bbB_r}^{\bbG_r}(\theta), r_{\bbT_r,\bbB_r}^{\bbG_r}(\vec \phi'; \vec \bfP') \rangle_{\bbG_r^\sigma} = \sum_{w \in W_{\bbG_r}(\bbT_r,\bbT_r')^\sigma} \langle \theta, \ad(w) \theta' \rangle_{\bbT_r^\sigma}.
  \end{equation*}
\end{proposition}

\begin{proof}
  We induct on the length $d'$ of $\vec \phi'$. The base case is $d' = 0$. We have
  \begin{align*}
    \langle R_{\bbT_r,\bbB_r}^{\bbG_r}(\theta), r_{\bbT_r'}^{\bbG_r}(\vec \phi'; \bfP') \rangle_{\bbG_r^\sigma} 
    &= \langle R_{\bbT_r,\bbB_r}^{\bbG_r}(\theta), \Inf_{\bbG_{r_0'}^\sigma}^{\bbG_r^\sigma}(R_{\bbT_{r_0'}',\bbB_{r_0'}'}^{\bbG_{r_0}}(\phi_{-1}')) \otimes \phi_0' \rangle_{\bbG_r^\sigma} \\
    &= \langle R_{\bbT_r,\bbB_r}^{\bbG_r}(\theta) \otimes \phi_0'{}^{-1}, \Inf_{\bbG_{r_0'}^\sigma}^{\bbG_r^\sigma}(R_{\bbT_{r_0'}',\bbB_{r_0'}'}^{\bbG_{r_0}}(\phi_{-1}')) \rangle_{\bbG_r^\sigma} \\
    &= \langle R_{\bbT_r,\bbB_r}^{\bbG_r}(\theta \otimes \phi_0'{}^{-1}), \Inf_{\bbG_{r_0'}^\sigma}^{\bbG_r^\sigma}(R_{\bbT_{r_0'}',\bbB_{r_0'}'}^{\bbG_{r_0}}(\phi_{-1}')) \rangle_{\bbG_r^\sigma} \\
    &= \langle R_{\bbT_{r_0},\bbB_{r_0}}^{\bbG_{r_0}}(\theta \otimes \phi_0'{}^{-1}), R_{\bbT_{r_0'}',\bbB_{r_0'}'}^{\bbG_{r_0}}(\phi_{-1}') \rangle_{\bbG_{r_0}^\sigma}.
  \end{align*}
  By construction, $\phi_{-1}$ is weakly $(\bfT,\bfG)$-generic of depth $r_0$, so we may apply the generic Mackey formula (Corollary \ref{cor:generic mackey}, which in this special case is the same as \cite[Theorem 1.1]{CI21-RT}) to obtain
  \begin{align*}
    \langle R_{\bbT_r,\bbB_r}(\theta), r_{\bbT_r'}^{\bbG_r}(\vec \phi'; \vec \bfP') \rangle_{\bbG_r^\sigma} 
    &= \sum_{w \in W_{\bbG_{r_0}}(\bbT_{r_0},\bbT_{r_0}')^\sigma} \langle \theta \otimes \phi_0'{}^{-1}, \ad(w^{-1})^* \phi_{-1}' \rangle_{\bbT_{r_0}^\sigma} \\
    &= \sum_{w \in W_{\bbG_{r}}(\bbT_{r},\bbT_{r}')^\sigma} \langle \theta, \ad(w^{-1})^* \theta' \rangle_{\bbT_{r}^\sigma}
  \end{align*}
  where in the last equality note that since $\phi_0'$ is a character of $\bbG_r^\sigma$, it is obviously invariant under pullback by $\ad(w)$.

  Now assume that the proposition holds for any $\theta'$ with Howe factorization length $d'$; we must show that the proposition holds for $\vec \phi'$ of length $d'+1$. We have
  \begin{align*}
    \langle R_{\bbT_r,\bbB_r}^{\bbG_r}(\theta), r_{\bbT_r,\bbB_r}^{\bbG_r}(\vec \phi'; \vec \bfP') \rangle_{\bbG_r^\sigma} 
    &= \langle R_{\bbT_r,\bbB_r}^{\bbG_r}(\theta), \Inf_{\bbG_{s_{d'}}^\sigma}^{\bbG_{r}^\sigma}(R_{\bbG_{s_{d'}}^{d'},\bbP_{s_{d'}}^{d'}}^{\bbG_{s_{d'}}}(r_{\bbT_{s_{d'}}'}^{\bbG_{s_{d'}}^{\prime(d')}}(\vec \phi_{\leq d'}'; \vec \bfP_{\leq d'}'))) \otimes \phi_{d'+1}' \rangle_{\bbG_r^\sigma} \\
    &= \langle R_{\bbT_r,\bbB_r}^{\bbG_r}(\theta) \otimes \phi_{d'+1}^{\prime-1}, \Inf_{\bbG_{s_{d'}}^\sigma}^{\bbG_{r}^\sigma}(R_{\bbG_{s_{d'}}^{d'},\bbP_{s_{d'}}^{d'}}^{\bbG_{s_{d'}}}(r_{\bbT_{s_{d'}}'}^{\bbG_{s_{d'}}^{d'}}(\vec \phi_{\leq d'}'; \vec \bfP_{\leq d'}')))\rangle_{\bbG_r^\sigma} \\
    &= \langle R_{\bbT_r,\bbB_r}^{\bbG_r}(\theta \otimes \phi_{d'+1}^{\prime-1}), \Inf_{\bbG_{s_{d'}}^\sigma}^{\bbG_{r}^\sigma}(R_{\bbG_{s_{d'}}^{d'},\bbP_{s_{d'}}^{\prime (d')}}^{\bbG_{s_{d'}}}(r_{\bbT_{s_{d'}}}^{\bbG_{s_{d'}}^{d'}}(\vec \phi_{\leq d'}'; \vec \bfP_{\leq d'}')))\rangle_{\bbG_r^\sigma} \\
    &= \langle R_{\bbT_{s_{d'}},\bbB_{s_{d'}}}^{\bbG_{s_{d'}}}(\theta \otimes \phi_{d'+1}^{\prime-1}), R_{\bbG_{s_{d'}}^{d'},\bbP_{s_{d'}}^{\prime(d')}}^{\bbG_{s_{d'}}}(r_{\bbT_{s_{d'}}'}^{\bbG_{s_{d'}}^{d'}}(\vec \phi_{\leq d'}'; \vec \bfP_{\leq d'}'))\rangle_{\bbG_{s_{d'}}^\sigma},
  \end{align*}
  where the third equality holds by the twisting lemma (Proposition \ref{prop:twisting}) and the fourth equality holds by in invariants lemma (Lemma \ref{lem:invariants}). Applying the generic Mackey formula (Corollary \ref{cor:generic mackey}) now gives
  \begin{align*}
    &\langle R_{\bbT_r,\bbB_r}^{\bbG_r}(\theta), r_{\bbT_r'}^{\bbG_r}(\vec \phi'; \vec \bfP') \rangle_{\bbG_r^\sigma} \\
    &= \sum_{w \in \bbT_{s_{d'}} \backslash \cS(\bbT_{s_{d'}},\bbG_{s_{d'}}^{d'})^\sigma/\bbG_{s_{d'}}^{d'}} \langle R_{\bbT_{s_{d'}}, \bbB_{s_{d'}} \cap {}^w \bbG_{s_{d'}}^{d'}}^{\bbG_{s_{d'}}^{d'}}(\theta \otimes \phi_{d'+1}^{\prime-1}), \ad(w^{-1})^* r_{\bbT_{s_{d'}}'}^{\bbG_{s_{d'}}^{d'}}(\vec \phi_{\leq d'}'; \vec \bfP_{\leq d'}') \rangle_{{}^w (\bbG_{s_{d'}}^{d'})^\sigma}.
  \end{align*}
  By the inductive hypothesis, each summand on the right-hand side is equal to
  \begin{align*}
    \sum_{v \in W_{{}^w \bbG_{s_{d'}}^{\prime (d')}}(\bbT_{s_{d'}},\bbT_{s_{d'}}')^\sigma} {}&{} \langle \theta \otimes \phi_{d'+1}^{\prime-1}, \ad(v^{-1})^* \ad(w^{-1})^* \theta' \otimes \phi_{d'+1}^{\prime-1} \rangle_{\bbT_{s_{d'}}^{\prime\sigma}} \\
    &= \sum_{v \in W_{{}^w \bbG_r^{\prime(d')}}(\bbT_r,\bbT_r')^\sigma} \langle \theta, \ad(v^{-1})^* \ad(w^{-1})^* \theta' \rangle_{\bbT_r^{\prime\sigma}}.
  \end{align*}
  The desired formula in the proposition now follows.
\end{proof}

\section{Variations: the Drinfeld stratification}\label{sec:drinfeld}

The methods in this paper can be mildly modified to yield results on the cohomology of the Drinfeld stratification of $X_{\bbT_r,\bbB_r}^{\bbG_r}$.

\begin{definition}[Drinfeld stratification]\label{def:drinfeld}
  The \textit{Drinfeld stratum} of $X_{\bbT_r,\bbB_r}^{\bbG_r}$ associated to a Levi subgroup $\bfL$ of $\bfG$ which contains $\bfT$, is the disjoint union
  \begin{equation*}
    \bigsqcup_{\gamma \in \bbG_r^\sigma/(\bbL_r \bbG_{0+:r+})^\sigma} \gamma \cdot X_{\bbT_,\bbB_r}^{\bbL_r\bbG_r^+}, \qquad \text{where $X_{\bbT_r,\bbB_r}^{\bbL_r\bbG_r^+} \colonequals \{x \in \bbL_r \bbG_{0+:r+} : x^{-1}\sigma(x) \in \sigma(\bbU_r)\}$}.
  \end{equation*}
  It is stable under the natural $(\bbT_r^\sigma \times \bbG_r^\sigma)$-action on $X_{\bbT_r,\bbB_r}^{\bbG_r}$. Denote by 
  \begin{equation*}
    R_{\bbT_r,\bbB_r}^{\bbL_r\bbG_r^+} \from \cR(\bbT_r^\sigma) \to \cR((\bbL_r\bbG_{0+:r+})^\sigma)
  \end{equation*}
  the functor corresponding to $X_{\bbT_r,\bbB_r}^{\bbL_r\bbG_r^+}$ in analogy with Definition \ref{def:induction}.
\end{definition}

\begin{theorem}\label{thm:drinfeld}
  Let $(\theta,\bbT_r,\bbB_r)$ be split-generic and Howe-factorizable. For any $(\theta',\bbT_r',\bbB_r')$,
  \begin{equation*}
    \langle R_{\bbT_r,\bbB_r}^{\bbL_r\bbG_r^+}(\theta), R_{\bbT_r',\bbB_r'}^{\bbL_r\bbG_r^+}(\theta') \rangle_{(\bbL_r\bbG_{0+:r+})^\sigma} = \sum_{w \in W_{\bbL_r\bbG_r^+}(\bbT_r,\bbT_r')^\sigma} \langle \theta, \ad(w)^* \theta' \rangle_{\bbT_r^\sigma}.
  \end{equation*}
\end{theorem}

\begin{proof}
  The properties of parahoric Deligne--Lusztig induction presented in Section \ref{sec:parahoric lusztig} have direct analogues for $R_{\bbT_r,\bbB_r}^{\bbL_r\bbG_r^+}$ and the proofs go through with only notational changes. The same is the case for Section \ref{sec:generic mackey} and especially Theorem \ref{thm:generic mackey}, the generic Mackey formula. Here, note that the generalized Bruhat decomposition in Section \ref{subsec:generic mackey} should be intersected with $\bbL_r \bbG_r^+ \subset \bbG_r$. The crux then is to see that the fiber calculations in Section \ref{sec:depth compatibility}. But this is again straightforward---since the Drinfeld stratification on $X_{\bbT_r,\bbB_r}^{\bbG_r}$ is defined by pullback from a stratification on $X_{\bbT_0,\bbB_0}^{\bbG_0}$, the fiber cohomology calculations required to establish Theorem \ref{thm:level lower} for a Drinfeld stratum is a special case of Theorem \ref{thm:fiber cohomology}. (A particular case to keep in mind is the closed Drinfeld stratum $X_{\bbT_r,\bbB_r}^{\bbT_r\bbG_r^+}$. This stratum lies over the locus $\bbG_0^\sigma \subset X_{\bbT_0,\bbB_0}^{\bbG_0}$ which corresponds to $u = 1$ in the notation of Section \ref{subsec:fiber cohomology}. The proof of Proposition \ref{prop:|A| elliptic} proves that the fibers of $X_{\bbT_r,\bbB_r}^{\bbT_r\bbG_r^+} \to X_{\bbT_{r-1},\bbB_{r-1}}^{\bbT_{r-1}\bbG_{r-1}^+}$ in fact all have the \textit{same} cohomology, even before taking $\bbT_{r:r+}^\sigma$-fixed points. Moreover, this phenomenon does not happen for any other stratum.)
\end{proof}

\providecommand{\bysame}{\leavevmode\hbox to3em{\hrulefill}\thinspace}
\providecommand{\MR}{\relax\ifhmode\unskip\space\fi MR }
\providecommand{\MRhref}[2]{%
  \href{http://www.ams.org/mathscinet-getitem?mr=#1}{#2}
}
\providecommand{\href}[2]{#2}

\end{document}